\documentclass[10pt]{amsart}
\usepackage{amsmath,amssymb,amsfonts}
\usepackage{enumerate}
\usepackage[pagebackref=true]{hyperref}
\hypersetup{
   colorlinks=true,
    linkcolor=black, citecolor=black,
   filecolor=magenta,      
    urlcolor=black,
    pdftitle={Archimedean Asai Zeta Integrals and Cohomological Test Vectors},
    pdfpagemode=UseNone,
    }
\usepackage{color}
\usepackage[dvipsnames]{xcolor}
\definecolor{sz}{RGB}{115,45,2}

\newcommand{\tb}[1]{\textcolor{sz}{\bf #1}}

\usepackage{geometry}
\usepackage{mathtools}

\usepackage{graphicx}
\usepackage{verbatim}
\newcommand{\iddots}{\reflectbox{$\ddots$}}

\newcommand{\C}{\mathbf C}

\newcommand{\GL}{\operatorname{GL}}

\newcommand{\As}{\operatorname{As}}

\theoremstyle{plain}
\newtheorem{thm}{Theorem}[section]
\newtheorem{cor}[thm]{Corollary}
\newtheorem{lem}[thm]{Lemma}
\newtheorem{prop}[thm]{Proposition}

\newtheorem{defn}[thm]{Definition}

\theoremstyle{remark}

\newtheorem{remark}[thm]{Remark}

\DeclareMathOperator{\Hom}{Hom}

\DeclareMathOperator{\Ind}{Ind}

\begin{document}
\title{Archimedean Asai Zeta Integrals and Cohomological Test Vectors}
\author{Bingchen Lin}
\address[B.L.]{        
Sichuan University, 24 South Section 1, 1st Ring Road, Chengdu, 610065,\\ China.}
\email{malinbc@scu.edu.cn}

\author{Dongming She}
\address[D.S.]{Beijing Institute of Mathematical Sciences and Applications (BIMSA), No. 544 Hefangkou Village Huaibei Town, Huairou District, Beijing, 101408,\\
China}
\email{shedongming@bimsa.cn}


\subjclass[2020]{11F70;22E50;22E35; 20G25}

\begin{abstract}
We construct explicit test vectors in the minimal $U(n)$-types
of cohomological principal series representations of
$\mathrm{GL}_n(\mathbb C)$, together with Schwartz functions,
for which the Flicker zeta integrals equal the corresponding
Asai $L$-factors. Both the vectors and the Schwartz functions
are independent of $s$, with all normalization constants
given explicitly. We construct the vectors through a polynomial
realization of Schur modules. Schur orthogonality reduces the
Flicker integrals to Mellin transforms of Whittaker functions,
and Godement sections yield a recurrence for their Jacquet
integrals. The lower-rank vectors depend on the remaining
integration variables, preventing a direct reduction to
lower-rank Flicker integrals. We overcome this difficulty by
combining growth estimates derived from the recurrence with
the Asai functional equation. This proves that the holomorphic
quotient of the Flicker integral by its Asai $L$-factor is
constant. An explicit asymptotic calculation shows that this
constant is exactly one.
\end{abstract}
\maketitle

\tableofcontents
\clearpage
\allowdisplaybreaks[0]
\setlength{\emergencystretch}{3em}
\section{Notations and preliminaries}\label{read:notations}
Throughout this paper, we set
\begin{itemize}
    \item $G_n:=\mathrm{GL}_n$. $B_n=A_nN_n$ is the Borel subgroup of upper triangular matrices in $G_n$, where $A_n$ is the maximal torus and $N_n$ is the unipotent radical of $B_n$. For a matrix $g$, write
$g^{-t}$ for ${}^t(g^{-1})$.

    \item Set \[K_n=U(n):=\{g\in \GL_n(\mathbb{C}): g{^t}\bar{g}=1\}\] \[K_{n,\mathbb{R}}=O(n):=\{g\in \GL_n(\mathbb{R}): g{^t}g=1\}\] 
    We also denote them by $K$ and $K_{\mathbb{R}}$ respectively, if there is no ambiguity on the size $n$. 
    
    \item For a matrix $A$ over $\mathbb C$, $A^*={}^t\bar A$ is its Hermitian adjoint.
    
    \item  $W(G_n,A_n)$: the Weyl group of $G_n$. Then $W(G_n,A_n)\simeq S_n$. Denote the long Weyl group element of $W(G_n, A_n)$ by $w_n$. We choose its representative as
$\dot{w}_n=\begin{pmatrix}
        & & 1\\
        & \iddots & \\
        1 & & \\
    \end{pmatrix}$.
    \item $\mathcal{S}(\mathbb{R}^n)$: the space of Schwartz functions on $\mathbb{R}^n$. 
    
    \item Let $\Gamma(z)$ be the standard $\Gamma$-function, set \[
 \Gamma_{\mathbb{R}}(z)=\pi^{-z/2}\Gamma(z/2),
 \quad \text{ and }
 \Gamma_{\mathbb{C}}(z)=2(2\pi)^{-z}\Gamma(z).
\]
\item For a partition $\lambda=(\lambda_1,\ldots,\lambda_n)$,
let $\mathbb S_\lambda$ be its Schur functor and put
$|\lambda|=\sum_i\lambda_i$.  We occasionally identify a partition with
its Young diagram. For a linear map $A:V\to W$,
$\mathbb S_\lambda(A):\mathbb S_\lambda(V)\to\mathbb S_\lambda(W)$
denotes the induced linear map. We occasionally denote the highest weight module of weight $\lambda$ by $V_\lambda$.

\item For $1\le k\le n$ put
$\eta^{[k]}=(d_1-d_k,\ldots,d_{k-1}-d_k,0)$ and
$V_k=\mathbb S_{\eta^{[k]}}(\mathbb C^k)$.
The minimal $U(k)$-type is $\tau_k=\det^{d_k}\otimes V_k$.
Denote by $v_k^+$ the highest weight vector of norm 1, with respect to a given $U(k)$-invariant Hermitian form $\langle\cdot,\cdot\rangle$ on $\tau_k$.
We use the Hermitian forms induced by the tensor realizations below and
identify $v_k^+$ with the corresponding vector in $V_k$ using the standard
unit vector in the determinant line $\det^{d_k}$.

\item For $\mathbf z=(z_1,\ldots,z_m)\in \mathbb C^m$, an integral over
$(\mathbf z)$ means integration over a product of
vertical lines 
\[
 \mathcal C(\mathbf c)=\prod_{j=1}^m(c_j+\mathrm i\mathbb R),\qquad \mathbf c=(c_1,\ldots,c_m)\in\mathbb R^m,
\]
with specific $c_i\in \mathbb R$, where $\mathrm i^2=-1$. Each vertical line is oriented from
$c_j-\mathrm i\infty$ to $c_j+\mathrm i\infty$.
More precisely, with $\mathbf t\in\mathbb R^m$ and Lebesgue
measure $d\mathbf t=dt_1\cdots dt_m$, our convention is
\[
 \int_{(\mathbf z)}F(\mathbf z)\prod_{j=1}^m\frac{dz_j}{2\pi\mathrm i}
 :=\frac1{(2\pi)^m}\int_{\mathbb R^m}
F(\mathbf c+\mathrm i\mathbf t)\,d\mathbf t.
\]
The inequalities beside each formula specify the admissible
real constants $c_i=\Re z_i$; these constants remain fixed
throughout that integral. All such integrals are first used
in a domain of absolute convergence. An empty contour integral
($m=0$) means the integrand itself, with no integration.
The notation $(z)$ has the same meaning when $m=1$.

\end{itemize}

\section{Introduction}

Integral representations connect the analytic properties of automorphic
$L$-functions with periods of automorphic forms. Their arithmetic
applications require control of the local integrals, including those at
archimedean places. The test-vector problem asks for explicit local data,
independent of the complex variable, whose zeta integral equals the
corresponding $L$-factor.

For cohomological representations, vectors in the minimal compact type
enter the comparison between Whittaker models and relative Lie algebra
cohomology. Exact evaluation of their integrals determines normalization
constants in period relations. For Asai $L$-functions, this connection
appears in the work of Grobner--Harris--Lapid on Whittaker rational
structures and special values \cite{grobner-harris-lapid2016}, and in the
rationality results of Grobner--Lin and their application to the refined
Gan--Gross--Prasad conjecture \cite{grobner-lin2021}. We provide such an
evaluation for the cohomological principal series below, with specified
vectors and measures, for use in forthcoming arithmetic applications.

We recall the three constructions of local Asai factors. Let $E/F$
be a quadratic extension of local fields of characteristic zero and
let $\pi$ be an irreducible generic representation of $\GL_n(E)$.
The arithmetic factor is
$L(s,\As^+(\phi_\pi))$, where $\phi_\pi$ is the local Langlands
parameter and $\As^+(\phi_\pi)$ is its positive tensor induction
from $E$ to $F$. There are also two analytic constructions.
Flicker's zeta integrals \cite{flicker1988,flicker1993}, further
developed by Kable \cite{kable2004}, define the Rankin--Selberg
Asai factor; over a nonarchimedean field it is the normalized
generator of the fractional ideal spanned by these integrals.
The Langlands--Shahidi construction defines Asai factors through
local coefficients for unitary groups
\cite{shahidi1990,goldberg1994}.
The choice of sign in the Asai representation is understood to agree
in these constructions; throughout this paper we use $\As^+$.

The agreement of these factors is the outcome of several works.
Over $p$-adic fields, Anandavardhanan--Rajan
\cite{anandavardhanan-rajan2005} proved the equality of the
Rankin--Selberg and Langlands--Shahidi $L$-factors for square-integrable
representations, and Matringe \cite{matringe2011} established it for
general generic representations. Henniart \cite{henniart2010}
identified the Langlands--Shahidi $L$-factor with the arithmetic
factor. Shankman \cite[Theorem~1]{shankman2018} proved compatibility
with tensor induction for both $L$- and $\epsilon$-factors.
At archimedean places, the comparison of local coefficients with
Artin factors belongs to Shahidi's work \cite{shahidi1985}.
Beuzart-Plessis \cite{beuzart2021archimedean} established the
archimedean analytic theory of Flicker integrals, including
meromorphic continuation, holomorphy after division by the arithmetic
$L$-factor, and the functional equation with the arithmetic
$\epsilon$-factor. His comparison of $\epsilon$-factors also treats
nonarchimedean fields of characteristic zero. The character and
measure corrections in that functional equation are essential for
the exact normalization used here.

Explicit test vectors require further information beyond the
identification of the local factors. In the nonarchimedean setting,
Matringe's construction of essential Whittaker functions
\cite{matringe2013} provides a basic tool for evaluating local
integrals. Anandavardhanan--Matringe
\cite{anandavardhanan-matringe2017} prove that the essential vector
is a test vector for the standard distinguishing period of a
distinguished generic representation, and evaluate that period in
terms of local $L$-values. For distinguished cuspidal representations
over fields of odd residual characteristic,
Anandavardhanan--Kurinczuk--Matringe--S\'echerre--Stevens
\cite[Theorem~1.8]{akmss2021} construct explicit Whittaker and
Schwartz data whose Flicker integral is the Asai $L$-factor.

At archimedean places, Humphries--Jo
\cite[Theorem~6.6 and Proposition~6.7]{humphries-jo2024}
construct weak test vectors for Flicker integrals using archimedean
newforms: the resulting integral is an explicit polynomial multiple
of the Asai $L$-factor.

Beuzart-Plessis \cite[Introduction, after Theorem~1]{beuzart2021archimedean}
formulates the expectation that, for a generic irreducible
representation $\pi$ of $\GL_n(\mathbb C)$, there exist finite
families $W_i\in\mathcal W(\pi,\psi)$ and
$\Phi_i\in\mathcal S(\mathbb R^n)$, independent of $s$, such that
\[
 L(s,\pi,\As^+)=\sum_i Z(s,W_i,\Phi_i).
\]
Motivated by arithmetic applications, we establish a stronger
form of this expectation for the cohomological principal series
considered here: we construct an explicit test vector in the
minimal $U(n)$-type and an explicit Schwartz function, both
independent of $s$, for which a single Flicker zeta integral
equals the Asai $L$-factor exactly.

Here is the main result. Fix $n\ge2$, a dominant integral weight
$\lambda=(\lambda_1,\ldots,\lambda_n)$, and $m\in\mathbb Z$.
Set
\[
 d_i=2\lambda_i+n+1-2i-m,\qquad
 \chi_{d_i}(z)=\left(\frac z{|z|}\right)^{d_i},\qquad
 \pi=\Ind_{B_n(\mathbb C)}^{G_n(\mathbb C)}
       \bigotimes_{i=1}^n
       \bigl(\chi_{d_i}|\cdot|_{\mathbb C}^{m/2}\bigr).
\]
Induction is normalized. This is the essentially tempered
cohomological representation associated with the pure coefficient
system $E_{\lambda,m}$ of Proposition~\ref{cohfamily}.
The integers $d_i$ strictly decrease and have common parity
$\epsilon\in\{0,1\}$. Put
$b_{ij}=(d_i-d_j)/2$ and $N=\binom n2$.
Let $H_{n\epsilon}^{(n)}$ be the real $O(n-1)$-fixed harmonic
polynomial of degree $n\epsilon$ normalized by
$H_{n\epsilon}^{(n)}(e_n)=1$, and let
$d_{n\epsilon}=\dim\mathcal H_n^{n\epsilon}$.

\begin{thm}\label{intro:main}
For every $\lambda$ and $m$ as above, the minimal $U(n)$-type of
$\pi$ contains the harmonic $O(n)$-type $\mathcal H_n^{n\epsilon}$
with multiplicity one, and no harmonic type of smaller degree.
Let $f_{n,\epsilon}$ be the section on the $O(n-1)$-fixed line
in this harmonic component, defined and normalized by
\eqref{read:lowest-polynomials}.
Set
\[
 \begin{gathered}
 c=2^{n-1},\qquad
 A_{n,\epsilon}=(-1)^{N\epsilon}2^{n(n+1)/2}
\prod_{i<j}\Gamma_{\mathbb C}(1+b_{ij}),\qquad
 \Phi_{n,\epsilon}(x)
 =A_{n,\epsilon}d_{n\epsilon}
   H_{n\epsilon}^{(n)}(\sqrt c\,x)e^{-\pi c\|x\|^2}.
 \end{gathered}
\]
With the Jacquet integral \eqref{read:Jacquet} and the quotient
measure specified in \eqref{eq:Iwasawa-measure}, one has
\[ Z(s,W_{f_{n,\epsilon}},\Phi_{n,\epsilon})
 =L(s,\pi,\As^+)=\Gamma_{\mathbb R}(s+m+\epsilon)^n
      \prod_{i<j}\Gamma_{\mathbb C}(s+m+b_{ij}).
\]
The integral converges absolutely for $\Re s>-m$, and the identity
holds elsewhere by meromorphic continuation.
\end{thm}

The vector is a polynomial in flag minors in a Schur module, and the
Schwartz function is a zonal harmonic times a Gaussian; both are
independent of $s$. The theorem combines Theorems~\ref{thm:structure}
and~\ref{read:main} with the polynomial construction.

The main difficulty is the evaluation of the integral after the
vector has been constructed. Iwasawa decomposition and Schur
orthogonality reduce it to the torus Mellin integral
$M_{n,\epsilon}(s+m)$ after integrating the positive
center. Jacquet's Godement section \cite[Section~7]{jacquet2009}
reproduces the prescribed
section with an explicit scalar and yields the recurrence
\eqref{read:recursion} for its Jacquet integral.
However, the lower-rank vector
$v_a(t)=D_{n-1}\mathbb S_{\eta^{[n]}}(\mathcal B(t,a))v$
depends on the remaining integration variables. Iteration therefore
produces coupled Gaussian integrals, rather than a product involving
the already evaluated lower-rank Flicker integral. To our knowledge,
no direct induction extracting that lower-rank integral is available
for these prescribed cohomological vectors.

We resolve this difficulty by combining the Godement recurrence
with the Asai functional equation and estimates for the resulting
Gaussian integrals. Write
\[
 Q(s)=
 \frac{Z(s-m,W_{f_{n,\epsilon}},\Phi_{n,\epsilon})}
      {L(s-m,\pi,\As^+)}
 =\frac{M_{n,\epsilon}(s)}{T_{n,\epsilon}(s)}.
\]
By \cite[Theorem~3.5.2]{beuzart2021archimedean},
$Q$ extends holomorphically to all of $\mathbb C$.
Iteration of the recurrence expresses its numerator as a Gaussian
integral with a homogeneous polynomial factor. Integration in a
common scaling variable compares this integral with the scalar
case, whose Mellin transform is evaluated using the
Ishii--Stade convolution \cite{ishii-stade} in the normalization of
\cite[Proposition~6.1]{chen-cheng-theta}.
Stirling's formula then gives
\[
 |Q(s)|\le C\left(\frac{1+|s|}{\Re s}\right)^\rho,
 \qquad \Re s\ge\tfrac12,\qquad
 \rho=\frac{n^2+n-4}{4}.
\]
The same estimate holds for the quotient associated with the
contragredient. The functional equation of
\cite[Theorem~3.5.1]{beuzart2021archimedean}, together with the
Fourier transform of the chosen Schwartz function, relates the
two quotients by $Q(s)=c_\pi Q^\vee(1-s)$, where $c_\pi\ne0$
is independent of $s$. This gives polynomial growth on the whole
plane. Cauchy's estimates imply that $Q$ is a polynomial, and
the bound on the positive real axis makes it constant.

Finally, rescaling the iterated integral reduces its asymptotics to the
unique minimum of $\mathcal F=\mathcal V-\log\mathcal A$. Evaluation
of the flag polynomials there gives $\lim_{s\to+\infty}Q(s)=1$.
Thus the coupled recurrence, functional equation, growth estimates,
and normalization limit determine the integral without assuming the
desired Asai identity in the analytic estimates.

Our construction provides explicit test vectors for cohomological
principal series uniformly in all ranks, with every normalization
constant specified. Although the evaluation of their Flicker
integrals uses growth estimates, asymptotic analysis, and the Asai
functional equation, the recurrence \eqref{read:recursion} arising
from the prescribed data in the minimal $K$-type is essential.
Its iteration determines both the homogeneous polynomial and the
exact scalar in the integral representation. Integration in the
common scaling variable produces $\Gamma(Ns+\Omega)$, whose shift
matches the total shift in the Asai gamma factors and yields the
precise bound in Proposition~\ref{read:analytic-input} (2).
The exact recurrence constants, together with the polynomial's
value at the unique minimum, then give the limit $Q(s)\to1$
in part (3). These two analytic inputs, proved in
Section~\ref{read:analytic}, thus depend crucially on the explicit
construction in the minimal $K$-type.

The paper is organized as follows. After the notation in Section~1,
Section~3 recalls the Asai factors, Flicker integrals, and the results
of Beuzart-Plessis used in the proof. Section~4 describes the
$O(n)$-types in $\mathcal S(\mathbb R^n)$ and their finite vectors.
Section~5 identifies the lowest harmonic type in the minimal
$U(n)$-type, constructs its polynomial vector, and derives the
Mellin reduction and Godement recurrence.
Subsection~\ref{read:evaluation} proves the main identity from the
functional equation and Proposition~\ref{read:analytic-input}.
Section~\ref{read:analytic} proves parts (2) and (3) of that proposition:
it evaluates the scalar comparison integral, derives the growth
estimate, and computes the normalization limit.

\subsection*{Acknowledgement}This joint work was initiated during the International Symposium on Number Theory held at Sichuan University in Oct. 2025. Both authors would like to thank the organizers for their warm hospitality.

\section{Archimedean Asai L-functions and Flicker zeta integrals} 
\subsection{Archimedean Asai L-functions}
We use $|z|_{\mathbb C}=|z|^2=z\bar z$ and $|x|_{\mathbb R}=|x|$.
Write
\[
 W_{\mathbb C}=\mathbb C^\times,\qquad
 W_{\mathbb R}=\mathbb C^\times\sqcup j\mathbb C^\times,\qquad
 jzj^{-1}=\bar z,\quad j^2=-1.
\]
For an L-parameter $\phi:W_{\mathbb C}\to\GL(V)$, define its  Asai transfer on $V\otimes V$ by tensor inductions:
\[
 \As^+(\phi)(z)(v\otimes w)=\phi(z)v\otimes\phi(\bar z)w,
 \qquad
 \As^+(\phi)(j)(v\otimes w)=\phi(-1)w\otimes v.
\]
We also set
$\As^-(\phi)=\As^+(\phi)\otimes\omega_{\mathbb C/\mathbb R}$,
where $\omega_{\mathbb C/\mathbb R}$ is the quadratic character defined by the extension $\mathbb C$ over $\mathbb R$ via class field theory. It is trivial on $W_{\mathbb C}$
and takes the value $-1$ on $j$.  For an irreducible admissible
representation $\pi$ of $G_n(\mathbb C)$ with L-parameter $\phi_\pi$, put
\[
 L(s,\pi,\As^\pm):=L(s,\As^\pm(\phi_\pi)).
\]
Throughout this paper, we consider $L(s,\pi,\As^+)$ represented by the Flicker zeta integrals; see \cite{beuzart2021archimedean}.

Suppose we have a normalized irreducible principal series of $\GL_n(\mathbb C)$:
\[
 \pi=\Ind_{B_n(\mathbb C)}^{G_n(\mathbb C)}
       (\chi_1\boxtimes\cdots\boxtimes\chi_n),\qquad
 \chi_i(z)=\left(\frac z{|z|}\right)^{d_i}|z|_{\mathbb{C}}^{\nu_i},
 \quad d_i\in\mathbb Z,\quad \nu_i\in\mathbb C.
\]
Put $\epsilon_i\in\{0,1\}$ with
$\epsilon_i\equiv d_i\pmod2$.  Thus
$\phi_\pi=\bigoplus_i\chi_i$, and 
\[
 \As^+(\phi_\pi)\simeq
 \bigoplus_{i=1}^n
   |\cdot|_{\mathbb{R}}^{2\nu_i}\operatorname{sgn}^{\epsilon_i}
 \ \oplus\
 \bigoplus_{i<j}\Ind_{W_{\mathbb C}}^{W_{\mathbb R}}(\chi_i\chi_j^c),
\] where $\chi_i^c(z)=\chi_i(\bar z)$. Therefore
\begin{equation}\label{eq:Asai-general-principal-series}
 L(s,\pi,\As^+)=
 \prod_{i=1}^n\Gamma_{\mathbb R}(s+2\nu_i+\epsilon_i)
 \prod_{i<j}\Gamma_{\mathbb C}
 \left(s+\nu_i+\nu_j+\frac{|d_i-d_j|}{2}\right).
\end{equation}
For $\As^-$, replace $\epsilon_i$ by $1-\epsilon_i$ in the real
factors; the complex factors are unchanged.
In particular, if $\nu_i=m/2$ for every $i$, the integers $d_i$ are
decreasing and have common parity $\epsilon$, then
\[
 L(s,\pi,\As^+)=
 \Gamma_{\mathbb R}(s+m+\epsilon)^n
 \prod_{i<j}\Gamma_{\mathbb C}
 \left(s+m+\frac{d_i-d_j}{2}\right).
\]

\subsection{Flicker zeta integrals for Archimedean Asai L-functions}
Let
\[
 \psi_{\mathbb R}(x)=e^{2\pi ix}\quad(x\in\mathbb R),
 \qquad
 \psi(z)=e^{2\pi i\operatorname{Im}z}\quad(z\in\mathbb C).
\]
Then $\psi|_{\mathbb R}=1$ and
$\psi(z)=\psi_{\mathbb R}
\bigl(\operatorname{Tr}_{\mathbb C/\mathbb R}(z_0z)\bigr)$, $z_0=-i/2$.
We also write $\psi$ for the generic character
$u\mapsto\psi(\sum_{i=1}^{n-1}u_{i,i+1})$ of $N_n(\mathbb C)$. Let $\pi$ be an irreducible $\psi$-generic Casselman--Wallach
representation of $G_n(\mathbb C)$. For
$W\in\mathcal W(\pi,\psi)$ and $\Phi\in\mathcal S(\mathbb R^n)$,
define, whenever the integral converges,
\begin{equation}\label{read:Flicker-integral}
 Z(s,W,\Phi)=
 \int_{N_n(\mathbb R)\backslash G_n(\mathbb R)}
 W(g)\Phi(e_ng)|\det g|_{\mathbb R}^s\,dg.
\end{equation}
Here $e_n:=(0,\cdots, 0,1)\in \mathbb R^n$ is a row vector. Put
\[
 \widetilde W(g)=W(\dot w_ng^{-t}),\qquad
 \widehat\Phi(y)=\int_{\mathbb R^n}\Phi(x)e^{2\pi ixy^t}\,dx.
\]
Then $\widetilde W\in\mathcal W(\pi^\vee,\psi^{-1})$.
The following theorem is established by Beuzart--Plessis in \cite{beuzart2021archimedean}:
\begin{thm}[Beuzart--Plessis {\cite[Theorem~1]{beuzart2021archimedean}}]
\label{read:BP-Asai}
For $W\in\mathcal W(\pi,\psi)$ and
$\Phi\in\mathcal S(\mathbb R^n)$, the following hold.
\begin{enumerate}[(i)]
\item\label{read:BP-convergence}
The integral $Z(s,W,\Phi)$ converges absolutely for
$\Re s$ sufficiently large and admits meromorphic
continuation to $\mathbb C$.

\item\label{read:BP-functional-equation}
With $\omega_\pi$ denoting the central character of $\pi$,
\begin{equation}\label{read:BP-FE}
 \frac{Z(1-s,\widetilde W,\widehat\Phi)}
      {L(1-s,\pi^\vee,\As^+)}
 =\omega_\pi(z_0)^{n-1}
 |z_0|_{\mathbb C}^{\binom n2(s-1/2)}
 \lambda_{\mathbb C/\mathbb R}(\psi_{\mathbb R})^{-\binom n2}\times
 \varepsilon_{\mathrm{Art}}
       (s,\As^+(\phi_\pi),\psi_{\mathbb R})
 \frac{Z(s,W,\Phi)}{L(s,\pi,\As^+)}.
\end{equation}
Here $|z_0|_{\mathbb C}=1/4$ and
$\lambda_{\mathbb C/\mathbb R}(\psi_{\mathbb R})=i$.

\item\label{read:BP-holomorphy}
The quotient $Z(s,W,\Phi)/L(s,\pi,\As^+)$ extends
holomorphically to $\mathbb C$.
If $\pi$ is nearly tempered in the sense of
\cite[Section~3.1]{beuzart2021archimedean}, then for each
$s_0\in\mathbb C$ there are $W$ and $\Phi$ for which this
holomorphic extension is nonzero at $s_0$.
\end{enumerate}
\end{thm}

\section{\texorpdfstring{$O(n)$-structure of $\mathcal{S}(\mathbb{R}^n)$}{O(n)-structure of Schwartz space}}In this section, 
we study the structure of the Schwartz space
$\mathcal{S}(\mathbb{R}^n)$ under the right regular action of $O(n)$. Corollary \ref{finite-vectors} of Theorem \ref{O(n)-Schwartz structure} is used to construct the explicit Schwartz functions in company with the cohomological test vectors for the evaluation of the Flicker integrals.  
For $d\geq0$, put
\begin{equation}
 \mathcal{H}^d_n
 =\left\{P\in\mathbb{C}[x_1,\cdots, x_n]: P\text{  is homogeneous of degree }d,\
   \Delta_{n}P=0\right\},
 \qquad \Delta_n=\sum_{i=1}^n\partial_{x_i}^2.
\end{equation}
to be the space of harmonic polynomials of degree $d$. Each $\mathcal{H}_n^d$ is an irreducible $O(n)$-representation, and the
$\mathcal{H}_n^d$ are pairwise $O(n)$-inequivalent for different $d$. Define the space of rapidly decreasing sequences by
\begin{equation}\label{s-definition}
 s(\mathbb N_0)
 =\left\{(a_q)_{q\geq0}\in\mathbb C^{\mathbb N_0}:
 p_N(a):=\left(\sum_{q\geq0}(1+q)^{2N}|a_q|^2\right)^{1/2}<\infty
 \text{ for every }N\geq0\right\}, 
\end{equation}where $\mathbb N_0$ is the set of non-negative integers. 
It is a Fr\'echet space for the seminorms $p_N$. 

\begin{thm}\label{O(n)-Schwartz structure}
There is an $O(n)$-equivariant topological isomorphism
\begin{equation}\label{abstract}
 \mathcal S(\mathbb R^n)\simeq
 \widehat{\bigoplus}_{d\ge0}
       \bigl(\mathcal H_n^d\widehat\otimes s(\mathbb N_0)\bigr),
\end{equation}
where $O(n)$ acts trivially on $s(\mathbb N_0)$. Here the
right-hand side consists of families $v_{q,d}\in\mathcal H_n^d$
satisfying
\begin{equation}\label{rapid}
 \sum_{q,d\ge0}(1+4q+2d)^{2N}\|v_{q,d}\|^2<\infty
 \qquad(N\ge0),
\end{equation}
with the square roots of these sums as seminorms.
Consequently, the only irreducible $O(n)$-types are
$\mathcal H_n^d$, and
\begin{equation}\label{isotypic}
 \mathcal S(\mathbb R^n)[\mathcal H_n^d]
 \simeq\mathcal H_n^d\widehat\otimes s(\mathbb N_0).
\end{equation}
\end{thm}

\begin{proof}
Restriction to $S^{n-1}$ identifies $\mathcal H_n^d$ with the
eigenspace of $\Delta_{S^{n-1}}$ of eigenvalue $-d(d+n-2)$
\cite[Chapter~IV]{stein-weiss1971}. Hence
\begin{equation}\label{sphere}
 L^2(S^{n-1})=\widehat{\bigoplus}_{d\ge0}\mathcal H_n^d.
\end{equation}
The harmonic oscillator $L=-\Delta_n+|x|^2$ has a positive
self-adjoint realization on $L^2(\mathbb R^n)$ commuting with $O(n)$.
For $H\in\mathcal H_n^d$ and $q\ge0$, put
\begin{equation}\label{laguerre}
 \Psi_{q,d,H}(x)=H(x)L_q^{d+n/2-1}(|x|^2)e^{-|x|^2/2},
\end{equation}
where $L_q^\alpha(t)=t^{-\alpha}e^t(q!)^{-1}
\frac{d^q}{dt^q}(e^{-t}t^{q+\alpha})$.
The Laguerre equation
\begin{equation}\label{laguerre-equation}
 t(L_q^\alpha)''+(\alpha+1-t)(L_q^\alpha)'+qL_q^\alpha=0
\end{equation}
and the polar-coordinate formula for $\Delta_n$ give
\begin{equation}\label{eigenvalue}
 L\Psi_{q,d,H}=(4q+2d+n)\Psi_{q,d,H}.
\end{equation}
Choose an orthonormal basis $H_{d,j}$, $1\le j\le N_d=\dim\mathcal H_n^d$, and
$c_{q,d}>0$ so that $\psi_{q,d,j}=c_{q,d}\Psi_{q,d,H_{d,j}}$
has norm one. By \eqref{sphere} and completeness of the Laguerre
system, these functions form an orthonormal basis of $L^2(\mathbb R^n)$.
Thus
\begin{equation}\label{expansion}
 f=\sum_{q,d,j}a_{q,d,j}\psi_{q,d,j}.
\end{equation}
For $N\ge0$, the spectral domain $\operatorname{Dom}(L^N)$ consists
of these expansions with
$\sum_{q,d,j}(4q+2d+n)^{2N}|a_{q,d,j}|^2<\infty$.
The Hermite characterization of Schwartz space
\cite[Theorem~4.9(a), with $E=\mathbb C$]{kruse2021series} gives
\begin{equation}\label{osc-char}
 \mathcal S(\mathbb R^n)=\bigcap_{N\ge0}\operatorname{Dom}(L^N)
\end{equation}
with seminorms $\|L^Nf\|_2$. By Parseval,
\begin{equation}\label{parseval}
 \|L^Nf\|_2^2=\sum_{q,d,j}(4q+2d+n)^{2N}|a_{q,d,j}|^2.
\end{equation}
Since $1+4q+2d\le4q+2d+n\le n(1+4q+2d)$, these seminorms
are equivalent to \eqref{rapid}. The map
$f\mapsto(v_{q,d})$, $v_{q,d}=\sum_j a_{q,d,j}H_{d,j}$,
is therefore the asserted topological isomorphism. It is
$O(n)$-equivariant because the radial factors are invariant;
irreducibility and inequivalence of the harmonic spaces give
\eqref{isotypic}.
\end{proof}

Denote by $\mathcal{S}(\mathbb{R}^n)_{O(n)\textrm{-fin}}$ the subspace of $O(n)$-finite vectors in $\mathcal{S}(\mathbb{R}^n)$. We have

\begin{cor}[$O(n)$-finite Schwartz functions]\label{finite-vectors}
Define
\begin{equation}
 \mathcal{S}(\mathbb{R}_{\geq0})=
 \left\{b\in C^\infty([0,\infty)):
 \sup_{t\geq0}(1+t)^A|b^{(B)}(t)|<\infty
 \text{ for all integers }A,B\geq0\right\}.
\end{equation}
A function $f\in\mathcal{S}(\mathbb{R}^n)$ is $O(n)$-finite if and only if it can be written as a finite sum
\begin{equation}\label{finite-vector-form}
 f(x)=\sum_{\ell=1}^{M}H_\ell(x)b_\ell(|x|^2),
\end{equation}
where each $H_\ell$ is a homogeneous harmonic polynomial and each
$b_\ell\in\mathcal{S}(\mathbb{R}_{\geq0})$.  Equivalently,
\begin{equation}\label{finite-algebraic-sum}
 \mathcal{S}(\mathbb{R}^n)_{O(n)\textrm{-fin}}
 =\bigoplus_{d\geq0}^{\mathrm{alg}}
 \mathcal{S}(\mathbb{R}^n)[\mathcal{H}_n^d].
\end{equation}
\end{cor}

\begin{proof}
If $\Phi$ is $O(n)$-finite, its translates span a finite-dimensional
space $\mathcal S_\Phi$. The isotypic projection
$P_d\Phi=(\dim\mathcal H_n^d)\int_{O(n)}\overline{\chi_d(k)}R(k)\Phi\,dk$,
where $\chi_d$ is the character of $\mathcal H_n^d$,
belongs to $\mathcal S_\Phi$. The nonzero projections are linearly
independent, so only finitely many occur and $\Phi=\sum_{d\in S}P_d\Phi$
for a finite set $S$.

For each $d\in S$, use the basis above and put $\alpha=d+n/2-1$.
Dur\'an's generalized Laguerre expansion theorem
\cite[Theorem~D, p.~739]{duran1989stieltjes}, together with
\eqref{rapid}, gives convergence in every Schwartz seminorm of
\[
 b_{d,j}(t)=\sum_{q\ge0}c_{q,d}a_{q,d,j}L_q^\alpha(t)e^{-t/2}
 \in\mathcal S(\mathbb R_{\ge0}).
\]
Consequently,
\begin{equation}\label{fixed-type-radial-form}
 P_d\Phi(x)=\sum_{j=1}^{N_d}H_{d,j}(x)b_{d,j}(|x|^2).
\end{equation}
Summing over $S$ proves \eqref{finite-vector-form}.
Conversely, the translates of each $H_\ell(x)b_\ell(|x|^2)$ lie in
the finite-dimensional space
$\{H(x)b_\ell(|x|^2):H\in\mathcal H_n^{\deg H_\ell}\}$.
Thus any finite sum in \eqref{finite-vector-form} is $O(n)$-finite,
and \eqref{finite-algebraic-sum} follows.
\end{proof}

\section{Cohomological test vectors for Archimedean Asai L-functions}
We regard $G_n(\mathbb C)$ as a real reductive group, and write
$\mathfrak g$ for its complexified real Lie algebra.
By abuse of notation, an irreducible Casselman--Wallach representation
is identified with its Harish--Chandra module when taking relative
Lie algebra cohomology.
\begin{defn}
Let $F$ be a finite-dimensional algebraic representation of
\[
 (\operatorname{Res}_{\mathbb C/\mathbb R}G_n)_{\mathbb C}
 \simeq G_n(\mathbb C)\times G_n(\mathbb C).
\]
An irreducible representation $\pi$ is \emph{cohomological with
coefficients in $F$} if
$H^\bullet(\mathfrak g,K;\pi\otimes F^\vee)\ne0$.
In this paper a \emph{cohomological test vector} means a vector in the
minimal $K$-type of such a representation.
\end{defn}
The coefficient systems used below are specified by a dominant
integral weight $\lambda=(\lambda_1,\ldots,\lambda_n)$ and an integer
$m$.  Put
\[
 \lambda^*=(m-\lambda_n,\ldots,m-\lambda_1),\qquad
 F_{\lambda,m}(g)=\tau_\lambda(g)\otimes\tau_{\lambda^*}(\bar g).
\]
Thus the two highest weights satisfy
$\lambda_i+\lambda^*_{n+1-i}=m$.
We use this pure family of coefficient systems and the essentially
tempered principal series specified in the next proposition.

\subsection{Cohomological principal series and their minimal \texorpdfstring{$K$}{K}-types.}
\begin{prop}\label{cohfamily}
Put $a_i=\lambda_i+(n+1-2i)/2$.  With normalized induction, the
principal series
\[
 \pi=\Ind_{B_n(\mathbb C)}^{G_n(\mathbb C)}
       (\mu_1\boxtimes\cdots\boxtimes\mu_n),\qquad
 \mu_i(z)=z^{a_i}\bar z^{\,m-a_i},
\]
is irreducible, essentially tempered, generic, and cohomological
with coefficients in $F_{\lambda,m}$.
Here the displayed character means
$|z|^m(z/|z|)^{2a_i-m}$, so it is single-valued.
\end{prop}
\begin{proof}
This is the representation $J_\mu$ of
\cite[Section~2.4.2, equation~(2.13)]{raghuram2016critical},
for the pair of coefficient weights $(\lambda,\lambda^*)$.
That section states its irreducibility, essential temperedness,
genericity and nonzero cohomology.
\end{proof}

 For $l\in \mathbb{Z}$ and $z\in \mathbb{C}^\times$, set $\chi_l(z)=\left(\frac{z}{\vert z\vert}\right)^l$, where $\vert z\vert=\vert z\vert_\C^{1/2}=(z\bar{z})^{1/2}$. Given $d=(d_1,\cdots, d_n)\in \mathbb{Z}^n$ and $\nu=(\nu_1,\cdots, \nu_n)\in \mathbb{C}^n$,
define a character $\chi_d$ and $\eta_\nu$ on $A_n$ by
\[\chi_d(a)=\prod_{i=1}^n\chi_{d_i}(a_i),\quad \eta_\nu(a)=\prod_{i=1}^n\vert a_i\vert_{\mathbb C}^{\nu_i}\] Then \[\mu_i(z)=\eta_{\frac{m}{2}}(z)\chi_{d_i}(z), \quad \textrm{where}\quad d_i=2\lambda_i+(n+1-2i)-m.\] 
In other words, the cohomological principal series representation
\begin{equation}\label{CohPS}
    \pi\simeq \mathrm{Ind}_{B(\mathbb{C})}^{G(\mathbb{C})}\vert \cdot \vert_\mathbb{C}^{\frac{m}{2}}\chi_{d_1}\boxtimes\vert\cdot\vert_{\mathbb{C}}^{\frac{m}{2}}\chi_{d_2}\boxtimes\cdots\boxtimes \vert \cdot\vert_{\mathbb{C}}^{\frac{m}{2}}\chi_{d_n}
\end{equation}

As the sequence $(\lambda_i)_{1\le i\le n}$ is decreasing, the sequence $(d_i)_{1\le i\le n}$ is strictly decreasing. Moreover, observe that
\begin{equation}\label{eq:parity}
 d_i\equiv n+1-m\pmod 2\quad(1\leq i\leq n).
\end{equation}
Let $\epsilon\in\{0,1\}$ be this common parity. For later use put 
\begin{equation}\label{eq:ri}
 r_i=\frac{d_i-d_{i+1}}2
 =\lambda_i-\lambda_{i+1}+1\geq1,
 \qquad 1\leq i<n.
\end{equation}

\begin{equation}\label{eq: Ri}
R_i=r_i+r_{i+1}+\cdots+r_{n-1},
 \qquad R_n=0.
\end{equation}

Restricting the induced functions to $K=U(n)$ identifies
$\pi|_K$ with
\begin{equation}\label{torus-induction}
I(d):=\mathrm{Ind}_{A_n(\mathbb C)\cap K}^{K}\chi_d.
\end{equation}
By Frobenius reciprocity, the unique minimal
$K$-type is $\tau=\tau_d|_K$ with multiplicity one, where $\tau_d$ is the highest weight representation of $G_n(\mathbb{C})$ of dominant weight $d$.
By irreducibility of $\pi$, every nonzero continuous Whittaker
functional $\Lambda\in\operatorname{Hom}_{N_n(\mathbb C)}(\pi,\psi)$
defines an injective map $v\mapsto W_v$, where
$W_v(g)=\Lambda(\pi(g)v)$, so every nonzero vector of the
minimal $K$-type gives a nonzero Whittaker function. For $\mu_i=\chi_{d_i}|\cdot|_{\mathbb C}^{\nu_i}$, put $\mu=\otimes_i\mu_i$. A section $f$ in the normalized
induction satisfies
\begin{equation}\label{read:induction}
 f(bg)=\delta_{B_n(\mathbb C)}(b)^{1/2}\mu(b)f(g).
\end{equation}
  Let $du=\prod_{i<j}d\Re u_{ij}\,d\Im u_{ij}$ be the measure on $N_n(\mathbb{C})$, where $d\Re u_{ij}$ and $d\Im u_{ij}$ are real Lebesgue measures on $\mathbb{R}$.
  
The following standard properties on Whittaker functions are recalled in \cite[\S\S3.3 and~7.2]{jacquet2009}.
\begin{lem}\label{Jacquet-integral}
The Jacquet integral 
\begin{equation}\label{read:Jacquet}
 W_{f_{d,\nu}}(g)=\int_{N_n(\mathbb C)}
       f_{d,\nu}(\dot w_nug)\psi^{-1}(u)\,du,
\end{equation}
initially defined in a convergence chamber, extends holomorphically
in $\nu$ for holomorphic families of sections in the fixed compact
model. Evaluation at $g=1$ defines a nonzero continuous Whittaker
functional for every $\nu$. For the irreducible cohomological
representation $\pi$, every nonzero section gives a nonzero
Whittaker function.
\end{lem}

Throughout, \eqref{read:Jacquet} denotes its analytic continuation when
the integral is not absolutely convergent. It will be specialized to
$\nu_i=m/2$ for cohomological principal series in this paper.

We have the Iwasawa decomposition
$G_n(\mathbb R)=N_n(\mathbb R)A_n(\mathbb R)^+O(n)$, where $A_n(\mathbb{R})^+=\{\mathrm{diag}(a_1,\cdots,a_n): a_i\in \mathbb{R}_+^\times\}$.
Fix the Haar measure $dk$ on $O(n)$ so that $\mathrm{Vol}(O(n), dk)=1$. Choose the Lebesgue measure
$dn$ for $N_n(\mathbb R)$ and
$d^\times a=\prod_{i=1}^n da_i/a_i$ for $A_n(\mathbb R)^+$.
We normalize $dg$ by
\begin{equation}\label{eq:Iwasawa-measure}
    \int_{G_n(\mathbb R)} f(g)\,dg
 =\int_{N_n(\mathbb R)}\int_{A_n(\mathbb R)^+}
   \int_{O(n)} \delta_{B_n(\mathbb R)}(a)^{-1}f(nak)
   \,dk\,d^\times a\,dn .
\end{equation}
Since $\psi$ is trivial on $N_n(\mathbb R)$ and $e_n$ is fixed
by $N_n(\mathbb{R})$, the Flicker integral is
\[
 Z(s,W,\Phi)=\int_{A_n(\mathbb R)^+}
  \left(\int_{O(n)}W(ak)\Phi(e_nak)\,dk\right)
  (\det a)^s\delta_{B_n(\mathbb R)}(a)^{-1}\,d^\times a.
\]
Embed $O(n-1)$ in $O(n)$ by $h\mapsto\operatorname{diag}(h,1)$. Put $P_{O(n-1)}v=\int_{O(n-1)}\tau(h)v\,dh$, with $\mathrm{Vol}(O(n-1),dh)=1$. Write $\tau[\mathcal H_n^\ell]$ for the $O(n)$-isotypic
component of type $\mathcal H_n^\ell$.
\begin{lem}\label{v-lie-in-O(n-1)-inv}
Assume $v$ lies in the minimal $K$-type $\tau$ of $\pi$. Suppose $Z(s,W_v,\Phi)\neq 0$, then $ \tau^{O(n-1)}\ne0$.
If, in addition, $\Phi$ is right $O(n-1)$-invariant, then there exists $v\in \tau$ such that $0\ne P_{O(n-1)}v\in\tau^{O(n-1)}$ and
\[
 Z(s,W_{P_{O(n-1)}v},\Phi)
 =Z(s,W_v,\Phi).
\]Consequently, we may assume that the test vector $0\ne v\in \tau^{O(n-1)}$.
\end{lem}
\begin{proof}
For fixed positive diagonal $a$, the function
 $k\mapsto\Phi(e_nak)
 \in
 \mathrm{Ind}_{O(n-1)}^{O(n)}\mathbf1$,
since $e_nah=e_na$ for $h\in O(n-1)$.
Nonvanishing of $Z(s,W_v,\Phi)$ requires
$I_{v,\Phi}(a):=\int_{O(n)}W_v(ak)\Phi(e_nak)\,dk\ne0$
for some $a$. By Schur orthogonality and Frobenius reciprocity,
\[
 0\ne\Hom_{O(n)}
 \bigl(\tau,\mathrm{Ind}_{O(n-1)}^{O(n)}\mathbf1\bigr)
 \simeq\Hom_{O(n-1)}(\tau|_{O(n-1)},\mathbf1).
\]
Since $\tau|_{O(n)}$ is self-dual,
$\Hom_{O(n-1)}(\tau|_{O(n-1)},\mathbf1)
 \simeq\tau^{O(n-1)}$.
Thus $\tau^{O(n-1)}\ne0$, proving the first assertion.
Suppose $\Phi(xh)=\Phi(x)$ for $h\in O(n-1)$. Since $\operatorname{vol}(O(n-1))=1$, by linearity in $v$ and
the substitution $g\mapsto gh^{-1}$, we have
\[
 Z(s,W_{P_{O(n-1)}v},\Phi)
 =\int_{O(n-1)}Z(s,W_{\tau(h)v},\Phi)\,dh
 =Z(s,W_v,\Phi).
\]
As $\tau^{O(n-1)}\ne0$, there exists $v\in\tau$ such that $0\ne P_{O(n-1)}v\in \tau^{O(n-1)}$. Replace $v$ by $P_{O(n-1)}v$ we may therefore assume $0\ne v\in\tau^{O(n-1)}$.
\end{proof}

\subsection{\texorpdfstring{$O(n-1)$}{O(n-1)}-invariant subspace of the minimal \texorpdfstring{$K$}{K}-type of cohomological principal series}
For the $O(n-1)$-invariant Schwartz functions used below,
Lemma~\ref{v-lie-in-O(n-1)-inv} forces us to choose the test vector
in the $O(n-1)$-fixed subspace of the minimal $K$-type. In this section, we study the structure of the $O(n-1)$-invariant part of the minimal $K$-type using branching rules, and specify our choice of the cohomological test vector as well as the Schwartz function which are used later to compute the Flicker integral. 

The minimal $K$-type of (\ref{CohPS}) can be written in terms of Schur modules as 
\begin{equation}\label{eq:tau}
 \tau\simeq\mathbb{S}_{\eta}(\mathbb{C}^n)\otimes {\det}(\cdot)^{d_n},\quad \mathrm{\ where\ } \eta=(2R_1,2R_2,\ldots,2R_{n-1},0).
\end{equation} Here $R_i$ is defined in (\ref{eq: Ri}). Since $d_n=2a+\epsilon$ for some $a\in\mathbb{Z}$, the restriction of
$\det^{2a}$ to $O(n)$ is trivial.  Thus the action of $O(n)$ on
$\tau$ is the same as its action on $\mathbb{S}_{\eta^{(\epsilon)}}(\mathbb{C}^n)$, where
\[\eta^{(\epsilon)}
 :=\eta+\epsilon(1^n)
 =(2R_1+\epsilon,\ldots,2R_n+\epsilon).\] 

Given a partition $\lambda=(\lambda_1,\cdots,\lambda_n)$, 
the branching law for the pair
$(U(n), U(n-1))$ gives
\begin{equation}\label{eq:unitaryBranching}
 \mathbb{S}_\lambda(\mathbb{C}^n)|_{U(n-1)}
=\bigoplus_{\beta}\mathbb{S}_\beta(\mathbb{C}^{n-1}),
\end{equation}
where every summand occurs with multiplicity one and the weight $\beta=(\beta_1,\cdots,\beta_{n-1})$ satisfies the interlacing relationship:
\begin{equation}\label{eq:interlace}
\lambda_1\geq\beta_1\geq\lambda_2
 \geq\cdots\geq
 \lambda_{n-1}\geq\beta_{n-1}\geq\lambda_n.
\end{equation}
By \cite[Chapter~VII, \S3, (3.14)]{macdonald1995},
\begin{equation}\label{eq:CH}
 \dim \mathbb S_\beta(\mathbb C^{n-1})^{O(n-1)}
 =\begin{cases}
 1,&\text{if every }\beta_i\text{ is even},\\
 0,&\text{otherwise}.
 \end{cases}
\end{equation} Apply this to $\lambda=\eta^{(\epsilon)}$, then $\lambda_i=2R_i+\epsilon$.
Set $t_i=2R_i+\epsilon-\beta_i$, then \eqref{eq:interlace} and \eqref{eq:CH} imply
$t_i\equiv\epsilon\pmod2$. 
Alternatively, we may write
$t_i=2j_i+\epsilon$,
 $0\leq j_i\leq r_i-\epsilon$
 $1\leq i<n$, and
define
\begin{equation}\label{eq:indexSet}
 \mathcal J_{n,\epsilon}
 =\left\{\mathbf j=(j_1,\ldots,j_{n-1}):
 0\leq j_i\leq r_i-\epsilon\right\}.
\end{equation}
The even interlacing partition attached to $\mathbf j$ is
\begin{equation}\label{eq:betaJ}
 \beta_{\mathbf j}
 =\bigl(2(R_1-j_1),2(R_2-j_2),\ldots,
 2(R_{n-1}-j_{n-1})\bigr).
\end{equation}

\begin{prop}\label{prop:fixedDimension}
One has
\[
 \dim\tau^{O(n-1)}
 =|\mathcal J_{n,\epsilon}|
 =\prod_{i=1}^{n-1}(r_i+1-\epsilon).
\]
The $O(n-1)$-fixed line indexed by $\mathbf j$ lies in the
$U(n-1)$-summand
$\det^{d_n-\epsilon}\otimes\mathbb S_{\beta_{\mathbf j}}(\mathbb C^{n-1})$.
\end{prop}

\begin{proof}
The branching in \eqref{eq:unitaryBranching} is multiplicity-free.  By
\eqref{eq:CH}, precisely the summands with even highest weight contribute
one $O(n-1)$-fixed line.  Equation \eqref{eq:betaJ} gives a
bijection between such summands and $\mathcal J_{n,\epsilon}$. The proposition follows by counting
independent choices of the $j_i$.
\end{proof}

We use the following form of the Pieri rule to compute those
multiplicities. Put $V=\mathbb C^n$.

\begin{defn}\label{def:strips}
For partitions $\nu\subset\lambda$, the skew diagram
$\lambda/\nu$ is a \emph{horizontal $\ell$-strip} if
$|\lambda|-|\nu|=\ell$ and no two new boxes lie in one column.
Equivalently,
\begin{equation}\label{eq:horizontalStrip}
 \lambda_1\geq\nu_1\geq\lambda_2\geq\nu_2
 \geq\cdots\geq\lambda_n\geq\nu_n,
 \qquad |\lambda|-|\nu|=\ell.
\end{equation}
\end{defn}

\begin{thm}[Pieri rule {\cite[Chapter~I, \S5, (5.16)]{macdonald1995}}]\label{thm:PieriRule}
For $\ell\geq0$,
\begin{equation}\label{eq:PieriSym}
 \mathbb S_\nu V\otimes\operatorname{Sym}^{\ell}V
 \simeq
 \bigoplus_{\substack{\lambda\supset\nu\\
 \lambda/\nu\ {\rm horizontal}\ \ell{\rm\mbox{-}strip}}}
 \mathbb S_\lambda V.
\end{equation}
Every summand has multiplicity one.
\end{thm}

For $\mathbf j\in\mathcal J_{n,\epsilon}$ put
\begin{equation}\label{eq:ellJ}
 \ell(\mathbf j)=n\epsilon+2|\mathbf j|,
 \qquad |\mathbf j|=j_1+\cdots+j_{n-1}.
\end{equation}
Let $H_\ell^{(n)}$ be the $O(n-1)$-fixed harmonic polynomial
normalized by $H_\ell^{(n)}(e_n)=1$. 
\begin{thm}\label{thm:structure} The generating series of the  multiplicities of $\mathcal H_n^\ell$ in the minimal $U(n)$-type $\tau$ satisfy
\begin{equation}\label{eq:multiplicityGenerating}
 \sum_{\ell\ge0}
 \dim\operatorname{Hom}_{O(n)}(\mathcal H_n^\ell,\tau)T^\ell
 =T^{n\epsilon}\prod_{i=1}^{n-1}
     (1+T^2+\cdots+T^{2(r_i-\epsilon)}),
\end{equation} where $T$ is a variable. 
Put
$M_{n,J}=\operatorname{Hom}_{O(n)}(\mathcal H_n^{n\epsilon+2J},\tau)$ for each non-negative integer $J$.
We obtain an isomorphism
\begin{equation}\label{eq:fixedMultiplicity}
 \tau^{O(n-1)}\simeq
 \bigoplus_{J\ge0}
       (\mathcal H_n^{n\epsilon+2J})^{O(n-1)}\otimes M_{n,J}.
\end{equation}
Choose a basis $\{\iota_{\mathbf j}:|\mathbf j|=J\}$
of $M_{n,J}$, indexed by $\mathbf j\in\mathcal J_{n,\epsilon}$,
and put $u_{\mathbf j}=\iota_{\mathbf j}(H_{\ell(\mathbf j)}^{(n)})$. Then
\begin{equation}\label{eq:fixedBasis}
 \tau^{O(n-1)}
   =\bigoplus_{\mathbf j\in\mathcal J_{n,\epsilon}}
          \mathbb C u_{\mathbf j},\quad \operatorname{span}_{\mathbb C}
   \{\tau(k)u_{\mathbf j}:k\in O(n)\}
     \simeq\mathcal H_n^{\ell(\mathbf j)}.
\end{equation}In particular, 
\(\ell(\mathbf0)=n\epsilon\) and $\dim\operatorname{Hom}_{O(n)}(\mathcal H_n^{n\epsilon},\tau)=1$.
\end{thm}
\begin{proof}
Put $\lambda=\eta^{(\epsilon)}$ and $V=\mathbb C^n$.
The even determinant twist in \eqref{eq:tau} is trivial on
$O(n)$, so $\tau|_{O(n)}\simeq\mathbb S_\lambda V|_{O(n)}$.
Set
\[
 c_\ell=\dim
   (\mathbb S_\lambda V\otimes\operatorname{Sym}^{\ell}V)^{O(n)},
 \qquad
 m_\ell=\dim\operatorname{Hom}_{O(n)}(\mathcal H_n^\ell,\tau),
\]
and set $c_\ell=0$ for $\ell<0$.
The invariant bilinear form identifies each harmonic space
with its $O(n)$-dual. We have an $O(n)$-equivariant decomposition
$\operatorname{Sym}^{\ell}V\simeq\bigoplus_{r=0}^{\lfloor\ell/2\rfloor}
\mathcal H_n^{\ell-2r}$. Since each harmonic representation is self-dual, taking
$O(n)$-invariants after tensoring with $\tau$ gives
$c_\ell=\sum_{r=0}^{\lfloor\ell/2\rfloor}
\dim\operatorname{Hom}_{O(n)}(\mathcal H_n^{\ell-2r},\tau)
 =m_\ell+m_{\ell-2}+\cdots$
Subtracting the same identity for $\ell-2$ yields
$m_\ell=c_\ell-c_{\ell-2}$, with $c_j=0$ for $j<0$.

By Theorem~\ref{thm:PieriRule} and
\eqref{eq:CH}, applied in dimension $n$,
$c_\ell$ equals the number of partitions $\mu$ with all
coordinates even such that $\mu/\lambda$ is a horizontal
$\ell$-strip. It satisfies the interlacing relationship
$\mu_1\ge\lambda_1\ge\mu_2\ge\lambda_2\ge\cdots\ge\mu_n\ge\lambda_n$. 
Since $\lambda_i=2R_i+\epsilon$, such partitions are
uniquely parametrized by
\[
 \mu_1=\lambda_1+\epsilon+2a_0,\quad a_0\ge0,\qquad
 \mu_{i+1}=\lambda_{i+1}+\epsilon+2a_i,\quad
 0\le a_i\le r_i-\epsilon\quad(1\le i<n).
\]The size of the strip is
 $|\mu|-|\lambda|=n\epsilon+2\sum_{i=0}^{n-1}a_i$.
Thus
\[
 \sum_{\ell\ge0}c_\ell T^\ell
 =\frac{T^{n\epsilon}}{1-T^2}
    \prod_{i=1}^{n-1}(1+T^2+\cdots+T^{2(r_i-\epsilon)}).
\]
Multiplication by $1-T^2$ proves
\eqref{eq:multiplicityGenerating}.

For every irreducible harmonic type, complete reducibility
and Schur's lemma show that the evaluation map
\[
 \mathcal H_n^\ell\otimes
     \operatorname{Hom}_{O(n)}(\mathcal H_n^\ell,\tau)
 \longrightarrow \tau[\mathcal H_n^\ell],\qquad
 H\otimes\iota\longmapsto\iota(H),
\] where $\tau[\mathcal H_n^\ell]$ is the $\mathcal H_n^\ell$-isotypic component of $\tau$,
is an isomorphism. All other irreducible $O(n)$-types have
no $O(n-1)$-fixed vectors. Therefore, by taking $O(n-1)$-fixed vectors and summing
over $\ell$, we obtain \eqref{eq:fixedMultiplicity}.
Each $(\mathcal H_n^\ell)^{O(n-1)}$ is spanned by
$H_\ell^{(n)}$. The chosen bases of $M_{n, J}$
now give the decomposition of $\tau^{O(n-1)}$ in \eqref{eq:fixedBasis}. Since a nonzero intertwining operator from
an irreducible representation to another is injective, this proves the second part of \eqref{eq:fixedBasis}. Note that
the coefficient of $T^{n\epsilon}$ in
\eqref{eq:multiplicityGenerating} is one, and all
coefficients of lower degree vanish.
\end{proof}

\begin{remark}For each $J$, the lines $\mathbb C u_{\mathbf j}$ with
$|\mathbf j|=J$ arise from the chosen basis of $M_{n,J}$.
They need not coincide with the $O(n-1)$-fixed lines in the
$U(n-1)$-summands of Proposition~\ref{prop:fixedDimension}. However,
the lowest line $\mathbb C u_{\mathbf0}$ is independent of
the choice of embeddings since $\dim M_{n,0}=1$ by Theorem~\ref{thm:structure}.     
\end{remark}

\subsection{The lowest harmonic type and reduction to Mellin transform}
\label{read:compact}
By Theorem~\ref{thm:structure}, the lowest harmonic component
$\tau[\mathcal H_n^{n\epsilon}]$ occurs with multiplicity one.
Its $O(n-1)$-fixed line is $\mathbb C u_{\mathbf0}$, even though
$\tau^{O(n-1)}$ may have larger dimension. Choose $v_{n,\epsilon}$
on this line: it is $O(n)$-fixed when $\epsilon=0$, and belongs
to the unique copy of $\mathcal H_n^n$ when $\epsilon=1$.
The polynomial \eqref{read:lowest-polynomials} fixes its scale;
Theorem~\ref{read:lowest-type} gives its polynomial realization.
For notational simplicity, put
\begin{equation}\label{eq:bij-N-s-alpha-l}
     b_{ij}=(d_i-d_j)/2,\quad N=\binom n2,\quad
 \ell=n\epsilon.
\end{equation}

\begin{remark}
We choose the lowest harmonic component because it occurs with
multiplicity one, so its $O(n-1)$-fixed line is canonical.
Moreover, the corresponding harmonic factor in the Schwartz
function contributes $t^{n\epsilon}$ when integrating over the
positive center $t=a_n>0$, producing the gamma factor
$\Gamma\bigl(n(s+\epsilon)/2\bigr)$ in \eqref{read:period-M}.
The explicit polynomial \eqref{read:lowest-polynomials}
also permits an exact calculation of the normalization.
In general, a test vector need not necessarily lie in the multiplicity-one
harmonic component.
\end{remark}

\subsubsection{The integral over $O(n)$.}
Let $H_\ell^{(n)}$ be the zonal harmonic polynomial with
$H_\ell^{(n)}(e_n)=1$, and put $h_\ell=\dim\mathcal H_n^\ell$.
At this moment, given constants $A\in\mathbb C$ and $c>0$, we set
\[\Phi_{A,c}(x)=A h_\ell H_\ell^{(n)}(\sqrt c\,x)e^{-\pi c\|x\|^2}.\]

\begin{lem}\label{read:Schur}
For $a=\operatorname{diag}(a_1,\ldots,a_n)$ with $a_i>0$,
\begin{equation}\label{read:compact-result}
 \int_{O(n)}W_{f_{v_{n,\epsilon}}}(ak)\Phi_{A,c}(e_nak)\,dk
 =A c^{\ell/2}a_n^\ell e^{-\pi c a_n^2}W_{f_{v_{n,\epsilon}}}(a).
\end{equation}
\end{lem}

\begin{proof}
Put $v=v_{n,\epsilon}$, $\ell=n\epsilon$, and
$V_\ell^{\mathrm{harm}}
=\operatorname{span}_{\mathbb C}\{\tau(k)v:k\in O(n)\}$.
By Theorem~\ref{thm:structure}, the harmonic type
$\mathcal H_n^\ell$ occurs in $\tau$ with multiplicity one,
and its $O(n-1)$-fixed subspace is one-dimensional. Thus
\[
 V_\ell^{\mathrm{harm}}\simeq\mathcal H_n^\ell,
 \qquad
 (V_\ell^{\mathrm{harm}})^{O(n-1)}=\mathbb Cv.
\]
Choose an orthonormal basis $u_1,\ldots,u_{h_\ell}$ with
$u_1=v/\|v\|$. An equivariant isomorphism identifies
$\mathbb Cu_1$ with $\mathbb C H_\ell^{(n)}$; rescaling it,
we may arrange that $u_1$ maps to $H_\ell^{(n)}$.
Equip $\mathcal H_n^\ell$ with the invariant Hermitian form
$\langle\cdot,\cdot\rangle$ making this identification unitary,
and write $(R(k)P)(x)=P(xk)$.

Normalize the Haar measure by $\operatorname{vol}(O(n-1),dh)=1$, note that it is independent of the normalization $\operatorname{vol}(O(n),dk)=1$ used later for Schur orthogonality,
since $dh$ is normalized on the subgroup itself and is not
obtained by restricting $dk$. Since $H_\ell^{(n)}$ is a unit vector spanning the
$O(n-1)$-fixed subspace, we therefore have
\[
 \int_{O(n-1)}R(h)P\,dh
 =\langle P,H_\ell^{(n)}\rangle H_\ell^{(n)}
 \qquad(P\in\mathcal H_n^\ell).
\]
Using $e_nh=e_n$ and $H_\ell^{(n)}(e_n)=1$, evaluation at
$e_n$ yields
$P(e_n)=\langle P,H_\ell^{(n)}\rangle$.
Taking $P=R(k)H_\ell^{(n)}$, we obtain
\[
 \varphi(k):=\langle\tau(k)u_1,u_1\rangle
 =\langle R(k)H_\ell^{(n)},H_\ell^{(n)}\rangle
 =H_\ell^{(n)}(e_nk).
\]
Thus, under $O(n-1)\backslash O(n)\simeq S^{n-1}$, the matrix
coefficient is the restriction of the homogeneous harmonic
polynomial $H_\ell^{(n)}$ of degree $\ell$.
In particular, $\varphi(k)$ is real.
By Schur orthogonality, for $1\le r\le h_\ell$, we have
\[\int_{O(n)}\langle\tau(k)u_1,u_r\rangle\overline{\langle\tau(k)u_1,u_1\rangle}\,dk
 =\frac{\delta_{r1}}{h_\ell}.
\]

On the other hand, expand
$\tau(k)v=\|v\|\sum_{r=1}^{h_\ell}\langle\tau(k)u_1,u_r\rangle u_r$.
Multiply by $h_\ell H_\ell^{(n)}(e_nk)$ and integrate, we get
\begin{equation}\label{read:vector-projection}
h_\ell\int_{O(n)}\tau(k)v_{n,\epsilon}H_\ell^{(n)}(e_nk)\,dk=v_{n,\epsilon}.
\end{equation}
Apply $v'\mapsto W_{f_{v'}}(a)$ and use
\[
 \Phi_{A,c}(e_nak)=A h_\ell c^{\ell/2}a_n^\ell
 H_\ell^{(n)}(e_nk)e^{-\pi c a_n^2},
\]
which follows from homogeneity and $\|e_nk\|=1$.
This proves \eqref{read:compact-result}.
\end{proof}

\subsubsection{The Mellin transform}By our choice of measures as in \eqref{eq:Iwasawa-measure},
the quotient measure on $N_n(\mathbb R)\backslash G_n(\mathbb R)$ is
\[
 dg=\delta_{B_n(\mathbb R)}(a)^{-1}\prod_i\frac{da_i}{a_i}\,dk,
 \qquad
 \delta_{B_n(\mathbb R)}(a)=\prod_i a_i^{n+1-2i}.
\]
Since $\psi$ is trivial on $N_n(\mathbb R)$, the Flicker integral is
\begin{align}
 Z(s,W_{f},\Phi_{A,c})
 &=\int_{A_n(\mathbb R)^+}\int_{O(n)}
 W_{f}(ak)\Phi_{A,c}(e_nak)\,dk\,
             (\det a)^s\delta_{B_n(\mathbb R)}(a)^{-1}
                                      \prod_i\frac{da_i}{a_i}.
 \label{read:Iwasawa}
\end{align}

Let $f=f_{v_{n,\epsilon}}$ be the section associated to $v_{n,\epsilon}$, define 
\begin{equation}\label{read:reduced}
 F_f(a)=\delta_{B_n(\mathbb R)}(a)^{-1}(\det a)^{-m}W_f(a).
\end{equation}  Let $f_\nu$ have the same restriction to $K$ as $f$ but with
inducing exponents $\nu_1,\ldots,\nu_n$. Initially choose
$\nu$ in the convergence chamber of the Jacquet integral.
Since $\dot w_nua=(\dot w_na\dot w_n^{-1})\dot w_n(a^{-1}ua)$, 
the induction law \eqref{read:induction} gives
\[
 f_\nu(\dot w_nua)
 =\delta_{B_n(\mathbb C)}(a)^{-1/2}
   \left(\prod_i a_{n+1-i}^{2\nu_i}\right)
   f_\nu(\dot w_n(a^{-1}ua)).
\]
Substituting $u=ava^{-1}$ in the Jacquet integral, we have
$du=\delta_{B_n(\mathbb C)}(a)\,dv$. Hence
\[
 W_{f_\nu}(a)
 =\delta_{B_n(\mathbb R)}(a)
   \left(\prod_i a_{n+1-i}^{2\nu_i}\right)
   \int_{N_n(\mathbb C)}
      f_\nu(\dot w_nv)\psi^{-1}(ava^{-1})\,dv,
\]
where we used
$\delta_{B_n(\mathbb C)}(a)^{1/2}
=\delta_{B_n(\mathbb R)}(a)$.
By holomorphic continuation in $\nu$, this identity remains
valid at $\nu_i=m/2$, where the character factor becomes
$(\det a)^m$. Dividing by
$\delta_{B_n(\mathbb R)}(a)(\det a)^m$ as in
\eqref{read:reduced}, and renaming $v$ as $u$, gives
\begin{equation}\label{read:variable-character}
 F_f(a)=\int_{N_n(\mathbb C)}f(\dot w_nu)
 \exp\left(-2\pi i\sum_{i=1}^{n-1}
       \frac{a_i}{a_{i+1}}\operatorname{Im}u_{i,i+1}\right)\,du.
\end{equation}
The integral is understood by the same continuation as \eqref{read:Jacquet}. In particular $F_f(ta)=F_f(a)$ for $t>0$.

Put $x_i=a_i/a_{i+1}$ and $t=a_n$.  Then
\[
 a=t\operatorname{diag}(x_1\cdots x_{n-1},\ldots,x_{n-1},1),
 \quad \det a=t^n\prod_{i=1}^{n-1}x_i^i,
 \quad \prod_i\frac{da_i}{a_i}=\frac{dt}{t}\prod_{i<n}\frac{dx_i}{x_i}.
\]
The last equality follows from the triangular logarithmic change
of variables, whose determinant is one. Define, for $s\in\mathbb C$,
\begin{equation}\label{read:M}
 M_{n,\epsilon}(s)=\int_{(0,\infty)^{n-1}}
 F_f(\operatorname{diag}(x_1\cdots x_{n-1},\ldots,x_{n-1},1))\prod_{i<n}x_i^{is}\frac{dx_i}{x_i}.
\end{equation}
Insert Lemma~\ref{read:Schur} and \eqref{read:reduced} in
\eqref{read:Iwasawa}, with the Flicker variable replaced by $s-m$.
The integration in $t$ gives
\[
 c^{n\epsilon/2}\int_0^\infty
       t^{n(s+\epsilon)}e^{-\pi ct^2}\frac{dt}{t}
             =\frac12c^{-ns/2}\pi^{-\frac{n(s+\epsilon)}{2}}\Gamma(\frac{n(s+\epsilon)}{2}).
\]
We have reduced the Flicker integral to the following Mellin transform
\begin{equation}\label{read:period-M}
 Z(s-m,W_f,\Phi_{A,c})
      =\frac A2c^{-ns/2}\pi^{-\frac{n(s+\epsilon)}{2}}\Gamma(\frac{n(s+\epsilon)}{2})M_{n,\epsilon}(s).
\end{equation}

\subsubsection{The identity to be proved}
In the rest of the proof we will choose
\begin{equation}\label{read:Schwartz}
 c=2^{n-1},\qquad
 A=A_{n,\epsilon}=(-1)^{N\epsilon}2^{n(n+1)/2}
                 \prod_{i<j}\Gamma_{\mathbb C}(1+b_{ij}),\qquad
 \Phi_{n,\epsilon}=\Phi_{A,c}.
\end{equation}
Set
\begin{equation}\label{read:T}
 T_{n,\epsilon}(s)=
 \frac{(-1)^{N\epsilon}}{2^{n-1}}\pi^{N(1-s)}
 \frac{\Gamma(\frac{s+\epsilon}{2})^n}{\Gamma(\frac{n(s+\epsilon)}{2})}
\prod_{i<j}\frac{\Gamma(s+b_{ij})}{\Gamma(1+b_{ij})}.
\end{equation}
By the definitions of $\Gamma_{\mathbb R}$ and $\Gamma_{\mathbb C}$,
direct multiplication shows that
\begin{equation}\label{read:target-normalization}
 \frac{A_{n,\epsilon}}2 2^{-Ns}\pi^{-\frac{n(s+\epsilon)}{2}}\Gamma(\frac{n(s+\epsilon)}{2})
       T_{n,\epsilon}(s)
 =\Gamma_{\mathbb R}(s+\epsilon)^n
 \prod_{i<j}\Gamma_{\mathbb C}(s+b_{ij})
 =L(s-m,\pi,\As^+).
\end{equation}
Thus \begin{equation}\label{Flicker-Asai-equivalence}
    Z(s,W_{f_{v_{n,\epsilon}}},\Phi_{n,\epsilon})=L(s,\pi,\As^+)\quad\Longleftrightarrow\quad M_{n,\epsilon}=T_{n,\epsilon}
\end{equation}

The next subsection constructs the section $f=f_{v_{n,\epsilon}}$ and derives its
recurrence relations via Godement sections. Subsection~\ref{read:evaluation} uses that recurrence,
the functional equation of Archimedean Asai L-functions established by Beuzart--Plessis in \cite{beuzart2021archimedean}, and two analytic estimates to prove
the equality between $M_{n,\epsilon}$ and $T_{n,\epsilon}$. Section~\ref{read:analytic} proves those estimates.

\subsection{Polynomial model for test vectors and recurrences via Godement section}
\label{read:polynomial}In this section, we construct a polynomial realization of the cohomological test vector by realizing the minimal $K$-type $\tau$ through Cartan products
of exterior powers, then express its vectors using polynomials in row flag minors.
This gives an explicit vector $v_{n,\epsilon}$ on the lowest line of
Theorem~\ref{thm:structure}. We then use Godement section to prove its recurrence relations.

\subsubsection{The Cartan product and its flag polynomial realization}
\label{read:Schur-functor}
Let $V=\mathbb C^n$ have its standard basis and Hermitian form,
linear in the first variable. Put
$\eta=(2R_1,\ldots,2R_{n-1},0)$, so that
$\tau=\det^{d_n}\otimes\mathbb S_\eta(V)$.
We first construct $\mathbb S_\eta(V)$ and its $O(n-1)$-invariant vector
by Cartan products, and then provide their polynomial realizations.

The tensor product of highest-weight vectors in two irreducible
rational $\mathrm{GL}_n(\mathbb C)$-modules generates their
\emph{Cartan component}. We call the orthogonal projection onto
this component the \emph{Cartan product}. It is
$\mathrm{GL}_n(\mathbb C)$-equivariant: its $U(n)$-equivariance
extends to the complexified Lie algebra and hence to the connected group.

We use the following tensor realization of Schur functors. For
$E=\mathbb C^p$ and a partition $\lambda$ with $h\le p$ nonzero parts,
put $c_j=\lambda_j-\lambda_{j+1}$, with $\lambda_{h+1}=0$, and
\[
 \mathcal T_\lambda(E)=\bigotimes_{j=1}^h
 (\textstyle\bigwedge^jE)^{\otimes c_j},\qquad
 v^+_{\lambda,E}=\bigotimes_{j=1}^h
 (e_1\wedge\cdots\wedge e_j)^{\otimes c_j}.
\]
The Cartan component generated by this unit highest vector is
$\mathbb S_\lambda(E)$. For any linear map $A:E\to F$,
including singular maps, the tensor map
$\mathcal T_\lambda(A)=\bigotimes_j(\bigwedge^jA)^{\otimes c_j}$
restricts to $\mathbb S_\lambda(A):\mathbb S_\lambda(E)\to
\mathbb S_\lambda(F)$. Here $\mathbb S_\lambda(F)=0$ if $\dim F<h$,
and the empty partition gives the identity on $\mathbb C$.
With the induced Hermitian forms,
\[
 \mathbb S_\lambda(AB)=\mathbb S_\lambda(A)\mathbb S_\lambda(B),
 \quad \mathbb S_\lambda(cA)=c^{|\lambda|}\mathbb S_\lambda(A),
 \quad \mathbb S_\lambda(A^*)=\mathbb S_\lambda(A)^*.
\]

Write $E_j=\bigwedge^jV$ and
$e_I=e_{i_1}\wedge\cdots\wedge e_{i_j}$ for
$I=\{i_1<\cdots<i_j\}$. Set
\[
 \mathcal T_\eta(V)=\bigotimes_{j<n}E_j^{\otimes 2r_j},
 \qquad
 v^+=\bigotimes_{j<n}(e_1\wedge\cdots\wedge e_j)^{\otimes 2r_j}.
\]
Let
$\omega_j=(1,\ldots,1,0,\ldots,0)$ be the $j$-th fundamental weight,
then the highest weight of $E_j$ is $\omega_j$ and $2\sum_{j<n}r_j\omega_j=\eta$. The vector $v^+$ has norm one and spans the weight-$\eta$
space of $\mathcal T_\eta(V)$. Thus the Cartan component is
\[\mathcal C_\eta=\operatorname{span}_{\mathbb C}\{\mathcal T_\eta(g)v^+:g\in\mathrm{GL}_n(\mathbb C)\}\simeq\mathbb S_\eta(V).\] Let
$\mathrm{pr}^{\mathrm{Car}}_\eta:\mathcal T_\eta(V)\to\mathcal C_\eta$
be the Cartan product.

For even $d_n$, the lowest harmonic type is the trivial
representation. To construct its vector, put
\begin{equation}\label{eq:Qi}
 Q_j=\sum_{|I|=j}e_I\otimes e_I\in E_j^{\otimes2}
 \qquad(1\le j<n).
\end{equation}
Each $Q_j$ is $O(n)$-fixed because $\bigwedge^jk$ is real orthogonal.
Thus
\begin{equation}\label{eq:qeta}
 \xi_\eta=\bigotimes_{j<n}Q_j^{\otimes r_j},\qquad
 q_\eta=\mathrm{pr}^{\mathrm{Car}}_\eta(\xi_\eta)
\end{equation}
is $O(n)$-fixed. It is nonzero with normalization
$\langle q_\eta,v^+\rangle=
 \langle\xi_\eta,v^+\rangle
 =\prod_{j<n}
 \left(\sum_{|I|=j}
 \langle e_I,e_1\wedge\cdots\wedge e_j\rangle^2\right)^{r_j}
 =1$.

For $v\in\mathcal T_\eta(V)$, define the following polynomial by the matrix coefficient
\begin{equation}\label{cartan:coefficient}
 P_v(X):=\langle\mathcal T_\eta(X)v,v^+\rangle,\qquad X\in M_n(\mathbb C).
\end{equation}
Then by definition,
\begin{equation}\label{cartan:translation}
P_{\mathcal T_\eta(g)v}(X)=P_v(Xg).
\end{equation}
For $X\in M_n(\mathbb C)$ with rows $x_1,\ldots,x_n\in V^*$,
put $\Lambda_j(X)=x_1\wedge\cdots\wedge x_j$.
When $X$ is invertible, these wedges determine its row flag.
Their coordinates in the dual wedge basis are
\begin{equation}\label{eq:flagMinor}
 \Lambda_j(X)=\sum_{|I|=j}p_{j,I}(X)e_I^*,
 \qquad
 p_{j,I}(X)=\Lambda_j(X)(e_I)=\det X_{\{1,\ldots,j\},I},
\end{equation} where $X_{I, J}$ is the minor of $X$ formed by rows indexed by $I$ and columns indexed by $J$. 
We call the $p_{j, I}$'s the \emph{row flag minors}. Note that $\Lambda_j(X)$ and $p_{j, I}(X)$ are defined even if $X$ is singular.

Let $
\mathcal P_\eta
=\operatorname{span}_{\mathbb C}
\left\{\prod_{j<n}\prod_{a=1}^{2r_j}p_{j,I_{j,a}}(X)\right\}$, with the right action of $\mathrm{GL}_n(\mathbb{C})$ by $
(\sigma(g)P)(X)=P(Xg).$

\begin{prop}\label{read:polynomial-model}
The map $v\mapsto P_v$ identifies the Cartan component
$\mathcal C_\eta\simeq\mathbb S_\eta(V)$ with $\mathcal P_\eta$.
The polynomial associated to the highest weight vector $v^+$ is
\[
 P^+(X):=P_{v^+}(X)
       =\prod_{j<n}p_{j,\{1,\ldots,j\}}(X)^{2r_j}.
\]
For every $P\in\mathcal P_\eta$ and \(t=\operatorname{diag}(t_1,\ldots,t_n)\),
$P(tX)=\left(\prod_{i=1}^n t_i^{\eta_i}\right)P(X).$ Moreover, the map
\begin{equation}\label{read:compact-section}
 P\longmapsto f_P|_{K},\qquad f_P(k)=\det(k)^{d_n}P(k),
\end{equation} gives a $K$-equivariant isomorphism between $\mathcal{P}_\eta\otimes \det(\cdot)^{d_n}$ and the minimal $K$-type $\tau\simeq \mathbb{S}_\eta(V)\otimes \det(\cdot)^{d_n}$ of $\pi$.
\end{prop}

\begin{proof}
The identity
$\langle(\bigwedge^jX)e_I,e_1\wedge\cdots\wedge e_j\rangle=p_{j,I}(X)$
gives $P^+$ and maps elementary tensors to the monomials spanning $\mathcal P_\eta$.
Since $\mathcal C_\eta^\perp$ is $\mathrm{GL}_n(\mathbb C)$-stable and
orthogonal to $v^+$, we have
$P_v=P_{\operatorname{pr}^{\mathrm{Car}}_\eta v}$, first for invertible
$X$ and then for all $X$ by polynomiality. Hence
$\mathcal C_\eta\to\mathcal P_\eta$ is surjective,
equivariant by \eqref{cartan:translation}, and nonzero
since $P_{v^+}(1)=1$, so irreducibility makes it an isomorphism.

Each monomial has row-$i$ degree $\sum_{j\ge i}2r_j=\eta_i$.
Hence $P(tX)=t^\eta P(X)$ and, for $t\in U(1)^n$, $h\in K$,
\[
 f_P(tk)=\prod_i t_i^{d_i}f_P(k),\qquad
 R(h)f_P=f_{\det(h)^{d_n}\sigma(h)P}.
\]
These identities make \eqref{read:compact-section} a $K$-equivariant
map into $I(d)$, nonzero since $f_{P^+}(1)=1$. Its source is
irreducible of highest weight $\eta+d_n(1,\ldots,1)=d$, so its
image is precisely that irreducible $K$-type.

To verify minimality, use $C_K=\sum_{i,j}E_{ij}E_{ji}$, with $E_{ij}$
the standard matrix units. Its eigenvalue on the highest weight module $V_\lambda$ of weight $\lambda$ is
$c(\lambda)=\sum_i\lambda_i(\lambda_i+n+1-2i)$.
If $V_\lambda$ occurs in $I(d)$, Frobenius reciprocity makes $d$ a
weight, so $\lambda-d=\sum_{i<n}a_i\alpha_i$ with $a_i\ge0$ and
$\alpha_i$ the simple roots. For $\lambda\ne d$, dominance gives
\[
 c(\lambda)-c(d)=\sum_{i<n}a_i
 (\lambda_i-\lambda_{i+1}+d_i-d_{i+1}+2)>0.
\]
Thus $c(d)$ is the smallest Casimir eigenvalue among the
$K$-types occurring in $I(d)$, and only the type of highest
weight $d$ has this eigenvalue. By Frobenius reciprocity, $V_d$ occurs with multiplicity
$\dim V_d[d]=1$, where $V_d[d]$ is its subspace of highest-weight $d$. Consequently, the
$c(d)$-eigenspace in the $K$-finite part of $I(d)$ is precisely
$\tau$. This confirms the minimality of the constructed
$K$-type.
\end{proof}

\subsubsection{Polynomial realization of test vectors and their Iwasawa coordinates}
\label{read:lowest-model}
Use the complex bilinear form $(x,y)=\sum_i x_i y_i$ and its
induced form on exterior powers. Thus
$(x_1\wedge\cdots\wedge x_j,y_1\wedge\cdots\wedge y_j)
=\det((x_a,y_b))$. Retain the notations for $\Lambda_j$ in the last subsection, put
\begin{equation}\label{read:Gram}
 \mathcal G_j=(\Lambda_j,\Lambda_j)=\sum_{|I|=j}p_{j,I}^2.
\end{equation}
\begin{lem}\label{Cartan-product-qeta}
The polynomial associated to $q_\eta$ in \eqref{eq:qeta} is $P_{q_\eta}=\prod_{j<n}\mathcal G_j^{r_j}$.
\end{lem}
\begin{proof}
Since $\xi_\eta-q_\eta\in\mathcal C_\eta^\perp$,
by Proposition~\ref{read:polynomial-model}, we have
$P_{q_\eta}=P_{\xi_\eta}$. Substituting
$Q_j=\sum_{|I|=j}e_I\otimes e_I$ into
$\xi_\eta=\bigotimes_{j<n}Q_j^{\otimes r_j}$ and distributing the tensor products,
we obtain a sum of elementary tensors. By \eqref{cartan:coefficient},
each pair $e_I\otimes e_I$ contributes $p_{j,I}(X)^2$.
Since $v\mapsto P_v$ is linear, we obtain
\[
 P_{\xi_\eta}(X)
 =\prod_{j<n}\left(\sum_{|I|=j}p_{j,I}(X)^2\right)^{r_j}
 =\prod_{j<n}\mathcal G_j(X)^{r_j}.
\]
\end{proof}
When $\epsilon=1$, the harmonic type of lowest degree in $\tau$ is $\mathcal H_n^n$.
We construct $B_n(\Lambda_\bullet;e)$ to be harmonic of degree
$n$ in $e$, of degree two in each $\Lambda_j$, and to satisfy
$B_n(\Lambda_\bullet h;e)=\det(h)B_n(\Lambda_\bullet;he)$
for $h\in O(n)$. 
Multiplication by $\prod_{j<n}\mathcal G_j^{r_j-1}$
gives the required degree $2r_j$ in each $\Lambda_j$.
Since the $\mathcal G_j$ are $O(n)$-invariant and
$he_n^t=e_n^t$ for $h\in O(n-1)$, evaluation at $e_n^t$
then gives an $O(n-1)$-fixed section of harmonic degree $n$,
provided it is nonzero.

Fix the standard ordered basis $e_1,\ldots,e_n$ of
$\mathbb C^n$ and put $\mathrm{vol}=e_1\wedge\cdots\wedge e_n$.
Define the complex-linear
Hodge map $*: \bigwedge^j\mathbb C^n\rightarrow \bigwedge^{n-j}\mathbb C^n$ characterized by
$\alpha\wedge *\beta=(\alpha,\beta)\mathrm{vol}$ for
$\alpha,\beta\in\bigwedge^j\mathbb C^n$. One should distinguish it from
the notation $A^*$ for a matrix $A$, which still means its Hermitian adjoint.
Identify $(\mathbb C^n)^*$ with $\mathbb C^n$ by the standard complex bilinear
form, and use the induced identification on exterior powers.
Set $\Lambda_0=1$ and, for a column vector $e\in\mathbb C^n$, define
\begin{equation}\label{read:B-definition}
 b_j(e)=(\Lambda_j,\Lambda_{j-1}\wedge e)\ (j<n),\qquad
 b_n(e)=(*\Lambda_{n-1},e),\qquad B_n(\Lambda_\bullet;e)=\prod_{j=1}^n b_j(e).
\end{equation}
These expressions are polynomial in the entries of $X$ and
the coordinates of $e$, and therefore remain defined for every
$X\in M_n(\mathbb C)$, including when some $\Lambda_j(X)$ vanish.

\begin{lem}\label{read:B-harmonic}
$B_n(\Lambda_\bullet;e)$ is harmonic of degree $n$ in $e$ and
has degree two in each $\Lambda_j$, and we have \begin{equation}\label{B_n-O(n)-equivariance}
    B_n(\Lambda_\bullet k;e)=\det(k)B_n(\Lambda_\bullet;ke),\quad \text{for all}\quad k\in O(n).
\end{equation} Suppose $k\in O(n)$ has rows $k_1,\ldots,k_n$, and
$\Lambda_j=k_1\wedge\cdots\wedge k_j$ for $1\le j<n$, then
\begin{equation}\label{read:B-orthogonal}
 \mathcal G_j=1,\qquad B_n(\Lambda_\bullet;e)=\det(k)\prod_j(k_j,e).
\end{equation}
\end{lem}

\begin{proof}
Write $\Lambda_j=x_1\wedge\cdots\wedge x_j$. For $j<n$,
\begin{equation}\label{determinant-formula-bj(e)}
 b_j(e)=\det\bigl((x_a,y_b)\bigr)_{1\le a,b\le j},
 \qquad (y_1,\ldots,y_j)=(x_1,\ldots,x_{j-1},e).
\end{equation}
Expansion along the last column gives $b_j(e)=(u_j,e)$ with
$u_j\in\operatorname{span}\{x_1,\ldots,x_j\}$.
Substitution of $e=x_i$ for $i<j$ gives two equal columns, so
$(u_j,x_i)=0$. Also $u_n=*\Lambda_{n-1}$ is perpendicular to
$x_1,\ldots,x_{n-1}$. Hence $(u_i,u_j)=0$ for $i\ne j$ and
\[
 \Delta_e B_n=2\sum_{i<j}(u_i,u_j)
                  \prod_{a\ne i,j}(u_a,e)=0.
\]
The degree in $e$ is $n$. Each $\Lambda_j$ occurs in exactly two
successive factors, giving degree two in that wedge.

For $k\in O(n)$, replacing $x_a$ by $x_ak$ in
\eqref{determinant-formula-bj(e)} replaces $b_j(e)$ by $b_j(ke)$.
For the last factor, the identity
\[
 (\Lambda_{n-1}k)\wedge e^t
 =\det(k)\Lambda_{n-1}\wedge(ke)^t
\]
gives the additional factor $\det(k)$. Multiplication proves
\eqref{B_n-O(n)-equivariance}. Now suppose that $k\in O(n)$ has rows $k_1,\ldots,k_n$ and
$\Lambda_j=k_1\wedge\cdots\wedge k_j$.
Since $(k_a,k_b)=\delta_{ab}$,
$
 \mathcal G_j=(\Lambda_j,\Lambda_j)
 =\det\bigl((k_a,k_b)\bigr)_{1\le a,b\le j}=1.$
The formula \eqref{determinant-formula-bj(e)} for $b_j$ gives
$b_j(e)=(k_j,e)$ for $j<n$.
Moreover, $*\Lambda_{n-1}$ is a scalar multiple of $k_n$,
and the Hodge identity gives
\[
 (*\Lambda_{n-1},k_n)\mathrm{vol}
 =\Lambda_{n-1}\wedge k_n=\det(k)\mathrm{vol}.
\]
Since $(k_n,k_n)=1$, we have
$*\Lambda_{n-1}=\det(k)k_n$. Therefore
$B_n(\Lambda_\bullet;e)=\det(k)\prod_{j=1}^n(k_j,e)$,
which proves \eqref{read:B-orthogonal}.
\end{proof}

Define
\begin{equation}\label{read:lowest-polynomials}
 P_{n,0}=\prod_{j<n}\mathcal G_j^{r_j},\qquad
P_{n,1}=B_n(\Lambda_\bullet;e_n^t)\prod_{j<n}\mathcal G_j^{r_j-1}.
\end{equation}
Both belong to $\mathcal P_\eta$, since their degree in each $\Lambda_j$
is $2r_j$ and $r_j\ge1$. We henceforth identify $v_{n,\epsilon}$
with $P_{n,\epsilon}$ by Proposition~\ref{read:polynomial-model}. This gives a polynomial realization of the cohomological test vector. The following theorem gives a precise connection between this realization and the terms in the Jacquet integral.

\begin{thm}\label{read:lowest-type}
For $\epsilon=0,1$, the section
$f_{n,\epsilon}(k)=\det(k)^{d_n}P_{n,\epsilon}(k)$
spans the $O(n-1)$-fixed line in the unique $O(n)$-subrepresentation
of the minimal $K$-type isomorphic to $\mathcal H_n^{n\epsilon}$.
For $u\in N_n(\mathbb C)$, let $\dot w_nu=b(u)k(u)$ be its Iwasawa decomposition. Denote the rows of $\dot w_nu$ and $k(u)$ by $m_i$ and $k_i$ respectively. Let $b_i$ be the $i$-th positive diagonal entry of $b(u)$.
Then they satisfy the following recursive relations
\begin{equation}\label{read:reverse-GS}
 q_i=m_i-\sum_{j>i}\langle m_i,k_j\rangle k_j,\qquad
 b_i=\|q_i\|,\qquad k_i=q_i/b_i\quad(i=n,n-1,\ldots,1).
\end{equation}
Set $\Lambda_j(k(u))=k_1\wedge\cdots\wedge k_j$.
Then we have
\begin{equation}\label{read:cell-formula}
 f_{n,\epsilon}(\dot w_nu)
 =(-1)^{\frac{n(n-1)}{2}\epsilon}\prod_i b_i^{n+1-2i}
 \begin{cases}
  \displaystyle\prod_{j<n}\mathcal G_j(\Lambda_j(k(u)))^{r_j},&\epsilon=0,\\
  \displaystyle B_n(\Lambda_\bullet(k(u));e_n^t)
             \prod_{j<n}\mathcal G_j(\Lambda_j(k(u)))^{r_j-1},&\epsilon=1.
 \end{cases}
\end{equation}
\end{thm}

\begin{proof}For $\epsilon=0,1$, the
polynomial $P_{n,\epsilon}$ has degree $2r_j$ in each
$\Lambda_j$. By Proposition~\ref{read:polynomial-model}
$f_{n,\epsilon}$ lies in the minimal $K$-type.
When $\epsilon=0$, Lemma~\ref{Cartan-product-qeta} shows that
$f_{n,0}$ is nonzero and $O(n)$-fixed. For $\epsilon=1$ and $k\in U(n)$, put
\[
 F(k,e)=\det(k)^{d_n}B_n(\Lambda_\bullet(k);e)
             \prod_{j<n}\mathcal G_j(\Lambda_j(k))^{r_j-1}.
\]
Lemma~\ref{read:B-harmonic} gives $F(k,\cdot)\in\mathcal H_n^n$
. Moreover, since $d_n$ is odd, for $k\in U(n)$ and $h\in O(n)$ we have
\[
\det(kh)^{d_n}B_n(\Lambda_\bullet(kh);e)
 =\det(k)^{d_n}\det(h)^{d_n+1}
       B_n(\Lambda_\bullet(k);he)
=\det(k)^{d_n}B_n(\Lambda_\bullet(k);he).
\]
Thus $F(kh,e)=F(k,he)$ for $k\in U(n)$ and $h\in O(n)$. The polynomial
$B_n(\Lambda_\bullet;e)\prod_{j<n}\mathcal G_j^{r_j-1}$
has degree $2r_j$ in the coordinates of each $\Lambda_j$.
Its coefficients in $e$ therefore belong to $\mathcal P_\eta$.
By Proposition~\ref{read:polynomial-model}, as a polynomial in $e$ ,each coefficient
of $F(k,e)$ is a function of $k$
valued in $\tau$.
Consequently, the map
\[
 T:(\mathcal H_n^n)^*\longrightarrow\tau,\qquad
 (T\ell)(k)=\ell(F(k,\cdot))
\]
is well-defined. The $O(n)$-actions on $\mathcal H_n^n$ and its dual are $(h\cdot H)(e)=H(h^{-1}e)$, $(h\cdot\ell)(H)=\ell(h^{-1}\cdot H)$.
The identity $F(kh,e)=F(k,he)$ implies that
$T(h\cdot\ell)=R(h)T(\ell)$, so $T$ is $O(n)$-equivariant.

The map $\operatorname{ev}_{e_n^t}\in (\mathcal H_n^n)^*$ defined by evaluation at the column vector $e_n^t$ is $O(n-1)$-fixed,
and $T(\operatorname{ev}_{e_n^t})=f_{n,1}$.
Choose $k_0\in O(n)$ with last column
$n^{-1/2}(1,\ldots,1)^t$. Equation~\eqref{read:B-orthogonal}
gives $f_{n,1}(k_0)=n^{-n/2}$, so $T\ne0$.
Since $(\mathcal H_n^n)^*\simeq\mathcal H_n^n$ is irreducible,
$T$ is injective. Thus $f_{n,1}$ is $O(n-1)$-fixed and
generates a copy of $\mathcal H_n^n$.
Together with the case $\epsilon=0$,
Theorem~\ref{thm:structure} identifies $f_{n,\epsilon}$ with the unique $O(n-1)$-fixed lines corresponding to $\mathbf j=\mathbf 0$ in the minimal $K$-type in both parities of $\epsilon$.

Write $\dot w_nu=b(u)k(u)$, with $b(u)\in B_n(\mathbb{C})$
and $b_i=b(u)_{ii}>0$. Since the rows $k_i$ are orthonormal,
\[
 m_i=\sum_{j\ge i}b(u)_{ij}k_j,\qquad
 \langle m_i,k_j\rangle=b(u)_{ij}\quad(j\ge i).
\]
Consequently $q_i=b_i k_i$, which proves
\eqref{read:reverse-GS}. Moreover, as $\det b(u)\det k(u)=\det \dot{w}_n u=\det\dot{w}_n=(-1)^{\frac{n(n-1)}{2}}$, $\det k(u)\in\{\pm 1\}$, and $\det b(u)=\prod_i b_i>0$, we have
$\det b(u)=1$, and $\det k(u)=(-1)^{\frac{n(n-1)}{2}}$.
Since each diagonal entry $b_i$ is positive,
$\chi_{d_i}(b_i)=1$ and $\mu_i(b_i)=b_i^m$. Hence \eqref{read:induction} gives
\[
 \begin{aligned}
 f_{n,\epsilon}(\dot w_nu)
 &=\prod_i b_i^{n+1-2i}(\det b(u))^m
       \det(k(u))^{d_n}P_{n,\epsilon}(k(u))\\
 &=(-1)^{\frac{n(n-1)}{2}\epsilon}\prod_i b_i^{n+1-2i}
       P_{n,\epsilon}(k(u)),
 \end{aligned}
\]
because $d_n\equiv\epsilon\pmod2$.
Substituting the definition of $P_{n,\epsilon}$ proves
\eqref{read:cell-formula}.
\end{proof}

\subsubsection{Recurrence of test vectors by Godement section}
\label{pub:Godement}
Put $J=\binom{I_{n-1}}{0}$ and $J_0=\binom{0}{I_{n-1}}$. For a rank-$(n-1)$ section $f$
with inducing characters $\mu_1,\ldots,\mu_{n-1}$ and a Schwartz function $\phi\in\mathcal S(M_{n,n-1}(\mathbb C))$, define
\begin{equation}\label{read:Jacquet-column}
 \mathcal G(f,\phi)(g)=\mu_n(\det g)|\det g|_{\mathbb C}^{-(n-1)/2}
 \int_{G_{n-1}(\mathbb C)}\phi(g^{-1}Jh)f(h)
       \mu_n(\det h)^{-1}|\det h|_{\mathbb C}^{n/2}\,dh.
\end{equation}
This is Jacquet's construction of the Godement section in terms of upper triangular Borel subgroups. Indeed, in his formula
\cite[(7.1)]{jacquet2009} using lower triangular Borels, take the characters $\mu_i^{-1}$,
the section $f\circ\iota$, and the kernel $A\mapsto\phi(A^t)$,
where $\iota(g)=g^{-t}$. Evaluation at $\iota(g)$ followed by
$h\mapsto h^t$ gives \eqref{read:Jacquet-column}, since
$((h,0)\iota(g))^t=g^{-1}Jh^t$.

Set $E=\mathbb S_\eta(\mathbb C^{n-1})$ and let $\rho_E$ denote
its $\mathrm{GL}_{n-1}(\mathbb C)$-action. We realize $E$ as
\[
 E=V_{n-1}\otimes
 \mathbb C\omega_{n-1}^{\otimes(d_{n-1}-d_n)},\qquad
 \omega_j=e_1\wedge\cdots\wedge e_j.
\]
Define $D_{n-1}:E\to V_{n-1}$ by
$D_{n-1}(v\otimes\omega_{n-1}^{\otimes(d_{n-1}-d_n)})=v$.
Since $\|\omega_{n-1}\|=1$ and
$(\bigwedge^{n-1}h)\omega_{n-1}=\det(h)\omega_{n-1}$, it is unitary and
\begin{equation}\label{pub:determinant-line}
 D_{n-1}\rho_E(h)=\det(h)^{d_{n-1}-d_n}
 \mathbb S_{\eta^{[n-1]}}(h)D_{n-1}.
\end{equation}
Our unit highest vectors satisfy
$v_E^+=v_{n-1}^+\otimes\omega_{n-1}^{\otimes(d_{n-1}-d_n)}$,
$D_{n-1}v_E^+=v_{n-1}^+$, and
$\mathbb S_\eta(J)v_E^+=v_n^+$.

For $1\le j\le n$ and $v\in V_j$, let $f_{j,v}$ be the section
in the normalized principal series induced from
$\mu_1,\ldots,\mu_j$, where
$\mu_i=\chi_{d_i}|\cdot|_{\mathbb C}^{m/2}$, determined by
\begin{equation}\label{eq:compact-section}
     f_{j,v}(k)=\det(k)^{d_j}
 \left\langle\mathbb S_{\eta^{[j]}}(k)v,v_j^+\right\rangle,
 \qquad k\in U(j).
\end{equation}
Its extension to $G_j(\mathbb C)$ is given by
$f_{j,v}(bk)=\delta_{B_j(\mathbb C)}(b)^{1/2}
\bigl(\prod_{i=1}^j\mu_i(b_{ii})\bigr)f_{j,v}(k)$.
In particular, $f_{n,v}=f_{P_v}$ in the notation of
\eqref{read:compact-section}, and
$f_{n,v_{n,\epsilon}}=f_{n,\epsilon}$. For $u\in E\setminus\{0\}$ and $v\in V_n$, put
\[
 \phi_{u,v}(X)=e^{-2\pi\operatorname{tr}(X^*X)}
                  \langle v,\mathbb S_\eta(X)u\rangle,
 \qquad X\in M_{n,n-1}(\mathbb C).
\]
By definition of the rank-$(n-1)$ section and the unitary identity
\eqref{pub:determinant-line}, for $k_0\in U(n-1)$, we have
\begin{equation}\label{eq:U(n-1)-section}
 f_{n-1,D_{n-1}u}(k_0)=\det(k_0)^{d_n}
 \langle\rho_E(k_0)u,v_E^+\rangle.
\end{equation}

Let $h=z t k_0$ be its Iwasawa decomposition, and fix the normalization of the Haar measure as follows:
\begin{equation}\label{pub:Godement-measure}
 dh=2^{(n-1)(n+2)/2}\delta_{B_{n-1}(\mathbb C)}(t)^{-1}
          dz\,d^\times t\,dk_0,
 \qquad\operatorname{vol}(U(n-1), dk_0)=1.
\end{equation}
Here $dz$ is the product of the measures $d^2z_{ij}$, $d^\times t=\prod_{i=1}^{n-1}\frac{dt_i}{t_i}$. The factor $2^{(n-1)(n+2)/2}$ is chosen to absorb the powers of $2$ arising
from the Gaussian integrations below, so that the scalar
in \eqref{pub:Godement-scalar} is
$C_n=\prod_{i<n}\Gamma_{\mathbb C}(1+b_{in})$. For rank one, put $V_1=\mathbb C$ and $v_{1,\epsilon}=1$.

\begin{prop}\label{pub:Godement-identity}
For $f_{j,v}$ defined as in \eqref{eq:compact-section}, $j=n-1,n$, the integral
\eqref{read:Jacquet-column} is absolutely convergent and
\begin{equation}\label{pub:Godement-scalar}
 \mathcal G(f_{n-1,D_{n-1}u},\phi_{u,v})
   =C_n\frac{\|u\|^2}{\dim E}f_{n,v},
 \qquad C_n=\prod_{i<n}\Gamma_{\mathbb C}(1+b_{in}).
\end{equation}
For $u=D_{n-1}^{-1}v_{n-1,\epsilon}$ and $v=v_{n,\epsilon}$,
the Godement section is therefore the nonzero multiple
$C_n\|u\|^2(\dim E)^{-1}f_{n,\epsilon}$.
\end{prop}
\begin{proof}
For $b'=\left(\begin{smallmatrix}b_0&*\\0&b_n'\end{smallmatrix}\right)$,
one has $(b')^{-1}J=Jb_0^{-1}$. Substitution of $h=b_0h'$ in
\eqref{read:Jacquet-column} gives the character $\prod_i\mu_i(b'_{ii})$
and the modulus
\[
 \delta_{B_{n-1}(\mathbb C)}(b_0)^{1/2}
 |\det b_0|_{\mathbb C}^{1/2}|b_n'|_{\mathbb C}^{-(n-1)/2}
 =\delta_{B_n(\mathbb C)}(b')^{1/2}.
\]
Thus the Godement section has the required induction law.
The identity $\phi_{u,v}(k^{-1}X)=\phi_{u,\mathbb S_\eta(k)v}(X)$
shows that both sides of \eqref{pub:Godement-scalar} are
$U(n)$-equivariant in $\det^{d_n}\otimes V_n$.
It suffices to compare their values at $1$ for every $v$.

Write $h=ztk_0$ as in \eqref{pub:Godement-measure}. By
\eqref{eq:U(n-1)-section},
\begin{equation}\label{eq:rank(n-1)-cancellation}
 f_{n-1,D_{n-1}u}(h)\mu_n(\det h)^{-1}
 =\delta_{B_{n-1}(\mathbb C)}(t)^{1/2}
       \langle\rho_E(k_0)u,v_E^+\rangle,
\end{equation}
and
\begin{equation}\label{eq:rank-(n-1)-modulus-character}
 \delta_{B_{n-1}(\mathbb C)}(t)^{-1/2}
 |\det t|_{\mathbb C}^{n/2}=\prod_{i<n}t_i^{2i}.
\end{equation}
Put $A=\mathbb S_\eta(J)\rho_E(zt)$. Since $z$ fixes $v_E^+$,
$Av_E^+=(\prod_{i<n}t_i^{d_i-d_n})v_n^+$.
Schur orthogonality gives
\[
 \int_{U(n-1)}\langle v,A\rho_E(k_0)u\rangle
 \langle\rho_E(k_0)u,v_E^+\rangle\,dk_0
 =\frac{\|u\|^2}{\dim E}\langle v,Av_E^+\rangle.
\]
The remaining unipotent Gaussian integral is
\[
 \int_{N_{n-1}(\mathbb C)}e^{-2\pi\|zt\|^2}\,dz
 =2^{-(n-1)(n-2)/2}\prod_{i<n}t_i^{-2(i-1)}
       e^{-2\pi\sum_{i<n}t_i^2}.
\]
Combining this with \eqref{pub:Godement-measure}, we have
\[
 \mathcal G(f_{n-1,D_{n-1}u},\phi_{u,v})(1)
 =\frac{\|u\|^2}{\dim E}\langle v,v_n^+\rangle
\prod_{i<n}\left(4\int_0^\infty
         t_i^{2+2b_{in}}e^{-2\pi t_i^2}\frac{dt_i}{t_i}\right)=C_n\frac{\|u\|^2}{\dim E}f_{n,v}(1).
\]
The last equality holds since $f_{n,v}(1)=\langle v,v_n^+\rangle$, and each factor in the parenthesis equals $\Gamma_{\mathbb C}(1+b_{in})$.
This holds for every $v$. Using $K$-equivariance, we may replace $v$ by
$\det(k)^{d_n}\mathbb S_\eta(k)v$. So the same equality holds for all $k\in K$.
Finally, the $B_n(\mathbb C)$-equivariance of the sections on both hand sides of \eqref{pub:Godement-scalar}
extends the equality to all $g\in G_n(\mathbb C)$ by Iwasawa decomposition.

To justify absolute convergence before integration over $U(n-1)$,
set $y_{ij}=z_{ij}t_j$. Then
\[\left(\prod_{i<n}t_i^{2i}\right)d^\times t\,dz
=\prod_{i<n}t_i\,dt_i\prod_{i<j<n}d^2y_{ij}.\]
Since $g^{-1}$ is invertible and $\|Jh\|=\|h\|$, there exists a constant $c_g>0$, depending only on $g$, such that $\|g^{-1}Jh\|\ge c_g\|h\|$ for every $h$. Thus, uniformly in $k_0\in U(n-1)$, the remaining integrand
is bounded by a polynomial in $t_i,\Re y_{ij},\Im y_{ij}$
times
$\exp\left[-2\pi c_g^2\left(\sum_{i<n}t_i^2+\sum_{i<j<n}|y_{ij}|^2\right)\right]$.
This bound is integrable against the chosen measure,
both as any $t_i\to0$ and at infinity.
\end{proof}

\subsubsection{Unfolding of the Jacquet integral}
\label{upper:unfolding}
We first work with general inducing characters
$\mu_i=\chi_{d_i}|\cdot|_{\mathbb C}^{\nu_i}$.
For $1\le j\le n$, write $W_{j,v;\nu}$ for the Jacquet
integral \eqref{read:Jacquet} of the section $f_{j,v;\nu}$
induced from $\mu_1,\ldots,\mu_j$ and determined by
\begin{equation}
 f_{j,v;\nu}(k)
 =\det(k)^{d_j}
 \left\langle\mathbb S_{\eta^{[j]}}(k)v,v_j^+\right\rangle,
 \qquad k\in U(j).
\end{equation}
Here $v_j^+$ is our chosen unit highest vector. Put
$\nu'=(\nu_1,\ldots,\nu_{n-1})$ and
\begin{equation}\label{audit:deformed-scalar}
 C_n(\nu)=\prod_{i<n}
 \Gamma_{\mathbb C}(1+b_{in}+\nu_i-\nu_n).
\end{equation}
The Jacquet integrals converge absolutely in the chamber
$\Re\nu_1>\cdots>\Re\nu_n$ and extend holomorphically in
$\nu$ by \cite[Section~7.2]{jacquet2009}, after applying
the map $g\mapsto g^{-t}$ described above. The following proposition proves a recurrence relation of Whittaker functions for general principal series.

\begin{prop}\label{upper:unfolding-proposition}
Assume $\Re\nu_1>\cdots>\Re\nu_n$. Let $v\in V_n$ and
$a=\operatorname{diag}(a_1,\ldots,a_n)$ with $a_i>0$.
For $U\in N_n(\mathbb C)$ and
$t=\operatorname{diag}(t_1,\ldots,t_{n-1})$ with $t_i>0$, put
\begin{equation}\label{matrix U, A}
 A=a^{-1}UJ_0t,\qquad
 J_0=\binom{0}{I_{n-1}}
     =\dot w_nJ\dot w_{n-1}^{-1}.
\end{equation}
Then
\begin{equation}\label{audit:before-Gaussian}
\begin{split}
 W_{n,v;\nu}(a)
 ={}&\frac{(-1)^{(n-1)\epsilon}2^{(n-1)(n+2)/2}}{C_n(\nu)}
 (\det a)^{2\nu_n-n+1}
 \int_{(0,\infty)^{n-1}}(\det t)^{n-2\nu_n}
 \delta_{B_{n-1}(\mathbb C)}(t)^{-1}\\
 &\quad\cdot\left[\int_{N_n(\mathbb C)}e^{-2\pi\|A\|^2}
 W_{n-1,D_{n-1}\mathbb S_\eta(A^*)v;\nu'}(t)
 \psi(U)\,dU\right]d^\times t.
\end{split}
\end{equation}
The integral on the right is absolutely convergent.
\end{prop}

\begin{proof}
Fix $u\in E\setminus\{0\}$ and let
$\mathcal G=\mathcal G(f_{n-1,D_{n-1}u;\nu'},\phi_{u,v})$.
The calculation in Proposition~\ref{pub:Godement-identity},
with the general inducing exponents, gives
\begin{equation}\label{eq:Godement-identity}
    \mathcal G=C_n(\nu)\frac{\|u\|^2}{\dim E}f_{n,v;\nu}.
\end{equation}
The double integral obtained by substituting the Godement section
into the Jacquet integral is absolutely convergent in this chamber
by the argument in \cite[Section~7.2]{jacquet2009}, after the
change of conventions described above. Hence Fubini's theorem
justifies the following interchanges.

Write $q\in N_n(\mathbb C)$ as
\[
 q=\begin{pmatrix}1&z\\0&\xi\end{pmatrix},\qquad
 z\in\mathbb C^{1\times(n-1)},\quad
 \xi\in N_{n-1}(\mathbb C).
\]
Then $dq=dz\,d\xi$ and $\psi(q)=\psi(z_1)\psi(\xi)$.
Insert \eqref{read:Jacquet-column} into the Jacquet integral
for $W_{\mathcal G}(a)$ and substitute
$h=\dot w_{n-1}\xi h'$. We obtain
\[
 q^{-1}\dot w_nJh=\binom{-z}{I_{n-1}}h',\qquad
 \int_{N_{n-1}(\mathbb C)}
 f_{n-1,D_{n-1}u;\nu'}(\dot w_{n-1}\xi h')
 \psi^{-1}(\xi)\,d\xi
 =W_{n-1,D_{n-1}u;\nu'}(h').
\]
Since $\det\dot w_j=(-1)^{j(j-1)/2}$ and
$\mu_n(-1)=(-1)^\epsilon$, the determinant factors contribute
\[
 \mu_n(\det\dot w_n)\mu_n(\det\dot w_{n-1})^{-1}
 =(-1)^{(n-1)\epsilon}.
\]
Renaming $h'$ as $h$ yields the intermediate unfolding identity
\begin{equation}\label{read:unfolded-row}
\begin{split}
 W_{\mathcal G}(a)
 ={}&(-1)^{(n-1)\epsilon}(\det a)^{2\nu_n-n+1}
 \int_{G_{n-1}(\mathbb C)}\int_{\mathbb C^{n-1}}
 \phi_{u,v}\left(a^{-1}\binom{-z}{I_{n-1}}h\right)
 \psi^{-1}(z_1)\,dz\\
 &\hspace{12mm}\cdot W_{n-1,D_{n-1}u;\nu'}(h)
 \mu_n(\det h)^{-1}|\det h|_{\mathbb C}^{n/2}\,dh.
\end{split}
\end{equation}

To obtain \eqref{audit:before-Gaussian}, write $h=xtk_0$,
where $x\in N_{n-1}(\mathbb C)$ and $k_0\in U(n-1)$, and set
\[
 U=\begin{pmatrix}1&-zx\\0&x\end{pmatrix}.
\]
The map $(z,x)\mapsto U$ preserves Lebesgue measure and satisfies
$\psi(x)\psi^{-1}(z_1)=\psi(U)$. Moreover,
$a^{-1}\binom{-z}{I_{n-1}}h=Ak_0$, with $A$ as in
\eqref{matrix U, A}. Thus the Schwartz kernel becomes
\[
 \phi_{u,v}(Ak_0)
 =e^{-2\pi\|A\|^2}
 \left\langle\mathbb S_\eta(A^*)v,\rho_E(k_0)u\right\rangle.
\]
The relation \eqref{pub:determinant-line} and Whittaker
equivariance give
\[
 W_{n-1,D_{n-1}u;\nu'}(xtk_0)\mu_n(\det(xtk_0))^{-1}
 =\psi(x)(\det t)^{-2\nu_n}
 W_{n-1,D_{n-1}\rho_E(k_0)u;\nu'}(t).
\]
For $w\in E$, Schur orthogonality gives
\[
 \int_{U(n-1)}\langle w,\rho_E(k_0)u\rangle
 \rho_E(k_0)u\,dk_0
 =\frac{\|u\|^2}{\dim E}w.
\]
Applying the linear map $w\mapsto W_{n-1,D_{n-1}w;\nu'}(t)$
to this identity yields
\[
 \int_{U(n-1)}\langle w,\rho_E(k_0)u\rangle
 W_{n-1,D_{n-1}\rho_E(k_0)u;\nu'}(t)\,dk_0
 =\frac{\|u\|^2}{\dim E}W_{n-1,D_{n-1}w;\nu'}(t).
\]
Substitute these identities into \eqref{read:unfolded-row},
using $w=\mathbb S_\eta(A^*)v$ and the Haar measure
\eqref{pub:Godement-measure}. The measure contributes
$2^{(n-1)(n+2)/2}\delta_{B_{n-1}(\mathbb C)}(t)^{-1}$,
while $|\det h|_{\mathbb C}^{n/2}=(\det t)^n$.
Finally, divide by $C_n(\nu)\|u\|^2/\dim E$ using \eqref{eq:Godement-identity}. This proves \eqref{audit:before-Gaussian}.
\end{proof}

\subsubsection{A recurrence relation}Retain the notation $V_{j}=\mathbb S_{\eta^{[j]}}(\mathbb C^{j})$. We prove a recurrence relation for the function \eqref{read:variable-character} associated to the section $f_v$ with $v\in V_n$. We begin with the following lemma.
\begin{lem}\label{read:Gaussian-lemma}
For $p\in\mathbb C[z]$, $c>0$, and $\zeta\in\mathbb R$,
\begin{equation}\label{read:Gaussian-Fourier}
 \int_{\mathbb C}p(\bar z)
 e^{-2\pi c^2|z|^2+2\pi i\zeta\Im z}\,d^2z
 =
 \frac{e^{-\pi\zeta^2/(2c^2)}}{2c^2}
 p\left(\frac{\zeta}{2c^2}\right).
\end{equation}
\end{lem}

\begin{proof}
Writing $z=x+iy$, one computes the real Gaussian integrals and obtains
\[
\int_{\mathbb C}
 e^{b\bar z-2\pi c^2|z|^2+2\pi i\zeta\Im z}\,d^2z=
 \frac1{2c^2}
 \exp\left(-\frac{\pi\zeta^2}{2c^2}
                  +\frac{b\zeta}{2c^2}\right).
\]
For $|b|\le B$, \[|\left(\frac{\partial}{\partial b}\right)^r(e^{b\bar z-2\pi c^2|z|^2+2\pi i\zeta\Im z})|\le
 |z|^r e^{-2\pi c^2|z|^2+B|z|}.\] The upper bound is integrable on $\mathbb C$.
This validates differentiation under the integral.
Taking the $r$th derivative at $b=0$ proves the formula
for monomials $p(z)=z^r$. It extends linearly to every polynomial.
\end{proof}
For the section $f_{d,v}$ associated to $v$, write
$F_{d,v}:=F_{f_{d,v}}$ as in \eqref{read:reduced}.
In rank one, $F_{(d_1),v}(a_1)=v$ for $a_1>0$.

\begin{prop}\label{read:Jacquet-recursion}
Let $v\in V_n$ and
$a=\operatorname{diag}(a_1,\ldots,a_n)$ with $a_i>0$.
For $t=\operatorname{diag}(t_1,\ldots,t_{n-1})$, define the
$(n-1)\times n$ matrix $\mathcal B(t,a)$ by
\begin{equation}\label{read:bidiagonal}
 \mathcal B(t,a)_{ii}=\frac{a_i}{t_i},\qquad
 \mathcal B(t,a)_{i,i+1}=\frac{t_i}{a_{i+1}},
\end{equation}
with all other entries zero, and put
$v_a(t)=D_{n-1}\mathbb S_{\eta^{[n]}}(\mathcal B(t,a))v$.
Then
\begin{equation}\label{read:recursion}
 F_{d,v}(a)
 =(-1)^{(n-1)\epsilon}\kappa_n
 \int_{(0,\infty)^{n-1}}
 F_{(d_1,\ldots,d_{n-1}),v_a(t)}(t)
 e^{-\pi\|\mathcal B(t,a)\|^2}\,d^\times t,\quad \text{where}\quad
 \kappa_n
 =\frac{2^{n-1}\pi^{n-1+\sum_{i<n}b_{in}}}
 {\prod_{i<n}\Gamma(1+b_{in})}.
\end{equation}
The integral is absolutely convergent for fixed $a$.
\end{prop}

\begin{remark}
The Godement section gives the recurrence of Whittaker functions
\eqref{audit:before-Gaussian} of
Proposition~\ref{upper:unfolding-proposition} for arbitrary inducing exponents, under the weight assumptions above.
The recurrence relation \eqref{read:recursion} and its constant
$\kappa_n$ use the common exponent $\nu_i=m/2$ of the family
\eqref{CohPS}. This allows the determinant factors to cancel under
the normalization \eqref{read:reduced}, and specializes
$\kappa_n(\nu)$ to $\kappa_n$. Under the same weight and parity assumptions, we only need a common
inducing exponent $\nu_i=\nu_0\in\mathbb C$, rather than
cohomologicality of $\pi$, to obtain \eqref{read:recursion} after
normalizing by $F=\delta_{B(\mathbb R)}^{-1}(\det)^{-2\nu_0}W$
in each rank.

However, the recurrence relation \eqref{read:recursion}, together
with the specific choices of polynomial realization of test vectors, is crucial to the estimate in
Proposition~\ref{read:analytic-input}(2).
Its iteration yields a Gaussian integral with a homogeneous
polynomial of degree $2\Omega$, where $\Omega=\sum_{i<j}b_{ij}$.
Integration in the common scaling variable then produces
$\Gamma(Ns+\Omega)$, whose shift matches the total shift in the
gamma factors of $T_{n,\epsilon}(s)$. This cancellation gives
precisely the exponent $\rho=(n^2+n-4)/4$.
For arbitrary $\nu_i$,
\eqref{audit:before-Gaussian} and its consequence
\eqref{upper:deformed-recurrence} retain additional determinant
powers and do not directly yield this homogeneous comparison.
Thus the general identity alone does not establish the required
bound by our argument; the specific structure of
\eqref{read:recursion} is an essential analytic input.
\end{remark}

\begin{proof}
Put $\nu'=(\nu_1,\ldots,\nu_{n-1})$ and $N=\binom n2=\frac{n(n-1)}{2}$.
First assume $\Re\nu_1>\cdots>\Re\nu_n$, so that
\eqref{audit:before-Gaussian} is absolutely convergent. Recall that in \eqref{matrix U, A}, $A=a^{-1}UJ_0t$. Its entries give
\[  (A^*)_{ri}=\frac{t_r}{a_i}\bar U_{i,r+1},\qquad  \|A\|^2=  \sum_{i<n}\left(\frac{t_i}{a_{i+1}}\right)^2  +\sum_{i<j}\left(\frac{t_{j-1}}{a_i}\right)^2|U_{ij}|^2. \]
The Whittaker function in the integrand is a polynomial in the
entries $\bar U_{ij}$, by polynomiality of the Schur map and linearity
in its vector. Apply Lemma~\ref{read:Gaussian-lemma} successively with
$c=t_{j-1}/a_i$ and $\zeta=1$ for $j=i+1$, $\zeta=0$ otherwise.
Each integration contributes $a_i^2/(2t_{j-1}^2)$ and substitutes
\[
 \bar U_{ij}\longmapsto
 \begin{cases}a_i^2/(2t_i^2),&j=i+1,\\0,&j>i+1.\end{cases}
\]
Thus $A^*$ becomes the bidiagonal matrix $\mathcal C(t,a)$ with
$\mathcal C_{ii}=a_i/(2t_i)$ and
$\mathcal C_{i,i+1}=t_i/a_{i+1}$. The Fourier Gaussian factors combine
with the original Gaussian to give $e^{-2\pi\|\mathcal C(t,a)\|^2}$.
Consequently,
\begin{equation}\label{read:Gaussian-evaluation}
\begin{aligned}
 &\int_{N_n(\mathbb C)}e^{-2\pi\|A\|^2}
 W_{n-1,D_{n-1}\mathbb S_{\eta^{[n]}}(A^*)v;\nu'}(t)\psi(U)\,dU\\
 &=2^{-N}\prod_{i<j}\left(\frac{a_i}{t_{j-1}}\right)^2
 e^{-2\pi\|\mathcal C(t,a)\|^2}
 W_{n-1,D_{n-1}\mathbb S_{\eta^{[n]}}(\mathcal C(t,a))v;\nu'}(t).
\end{aligned}
\end{equation}
Since $\prod_{i<j}(a_i/t_{j-1})^2
=\prod_i a_i^{2(n-i)}\prod_{r<n}t_r^{-2r}$, the modulus factors satisfy
\begin{equation}\label{read:moduli}
 (\det a)^{-(n-1)}\prod_{i<j}\left(\frac{a_i}{t_{j-1}}\right)^2
 (\det t)^n\delta_{B_{n-1}(\mathbb C)}(t)^{-1}
 =\delta_{B_n(\mathbb R)}(a)\delta_{B_{n-1}(\mathbb R)}(t)^{-1}.
\end{equation}
Substitute \eqref{read:Gaussian-evaluation} into
\eqref{audit:before-Gaussian} and then put $t=q/\sqrt2$.
The multiplicative measure and real modulus are unchanged, while
\[
 \mathcal C(q/\sqrt2,a)=2^{-1/2}\mathcal B(q,a),\qquad
 e^{-2\pi\|\mathcal C(q/\sqrt2,a)\|^2}=e^{-\pi\|\mathcal B(q,a)\|^2}.
\]
Schur homogeneity, the lower-rank central character, and the determinant
power contribute, respectively,
$2^{-|\eta^{[n]}|/2}$, $2^{-\sum_{i<n}\nu_i}$, and
$2^{(n-1)\nu_n}$. Since $|\eta^{[n]}|/2=\sum_{i<n}b_{in}$,
renaming $q$ as $t$ gives
\begin{equation}\label{upper:deformed-recurrence}
\begin{aligned}
 W_{n,v;\nu}(a)
 ={}&(-1)^{(n-1)\epsilon}\kappa_n(\nu)
 \delta_{B_n(\mathbb R)}(a)(\det a)^{2\nu_n}\\
 &\quad\cdot\int_{(0,\infty)^{n-1}}
 \frac{W_{n-1,v_a(t);\nu'}(t)}
 {\delta_{B_{n-1}(\mathbb R)}(t)(\det t)^{2\nu_n}}
 e^{-\pi\|\mathcal B(t,a)\|^2}\,d^\times t,\\
 \kappa_n(\nu)
 ={}&\frac{2^{(n-1)(n+2)/2-N-\sum_{i<n}(b_{in}+\nu_i-\nu_n)}}{C_n(\nu)}
 =\frac{2^{n-1}\pi^{n-1+\sum_{i<n}(b_{in}+\nu_i-\nu_n)}}
 {\prod_{i<n}\Gamma(1+b_{in}+\nu_i-\nu_n)}.
\end{aligned}
\end{equation}

We now extend this identity to all $\nu$.
The Jacquet integrals depend holomorphically on the parameters
in the fixed Fr\'echet compact model
\cite[Section~7.2]{jacquet2009}. The uniform boundedness principle
therefore makes this family equicontinuous on compact parameter sets.
The Iwasawa decomposition expresses the right translate
by $t$ in terms of the section restricted to $U(n-1)$.
Each seminorm involving finitely many derivatives of this
translate is bounded by a polynomial in $t_i$ and $t_i^{-1}$
times a suitable seminorm of the original section.
These bounds are uniform on compact sets of inducing parameters.
Consequently, the following inequality holds uniformly on that parameter set:
\[
 |W_{n-1,w;\nu'}(t)|
 \le C\|w\|\prod_{i<n}(t_i+t_i^{-1})^M.
\]
For fixed $a$, the coordinates of $v_a(t)$ are Laurent
polynomials in $t$. After also including the determinant
and modulus characters, the absolute value of the integrand
in \eqref{upper:deformed-recurrence} is bounded by
\[
 C_a\|v\|\prod_{i<n}(t_i+t_i^{-1})^{M'}
 \exp\left[-\pi\sum_{i<n}
 \left(\frac{a_i^2}{t_i^2}+\frac{t_i^2}{a_{i+1}^2}\right)\right]
\]
for some constant $C_a$, depending on $a$ and the chosen compact
parameter set. This bound is integrable against $d^\times t$ and
independent of the parameters within that set.
Thus the integral on the right hand side of \eqref{upper:deformed-recurrence} is holomorphic in $\nu$. Since $\kappa_n(\nu)$ and  the Jacquet integral on the
left-hand side are entire in $\nu$, the identity
\eqref{upper:deformed-recurrence} extends holomorphically for every $\nu\in\mathbb C^n$.

Finally, set $\nu_i=m/2$ and apply the normalization
$F=\delta_{B(\mathbb R)}^{-1}(\det)^{-m}W$
in ranks $n$ and $n-1$.
The determinant and modulus factors cancel, yielding
\eqref{read:recursion}. Its
absolute convergence is obtained by the same bounds.
\end{proof}

\subsection{Evaluation by the Asai functional equation}\label{read:evaluation}
Consider the quotient:
\begin{equation}\label{read:Q}
 Q(s)=\frac{M_{n,\epsilon}(s)}{T_{n,\epsilon}(s)}
 =\frac{Z(s-m,W_{f_{n,\epsilon}},\Phi_{n,\epsilon})}
        {L(s-m,\pi,\As^+)}.
\end{equation}
The second equality follows from
\eqref{read:period-M} and \eqref{read:target-normalization}. We prove $Q(s)=1$ by combining the bounds of
Proposition~\ref{read:analytic-input} (2) with the Asai functional
equation \cite{beuzart2021archimedean} to obtain polynomial growth
on $\mathbb C$. Cauchy's estimates then make $Q$ a polynomial,
and boundedness on the positive real axis makes it constant.
Proposition~\ref{read:analytic-input} (3) identifies this constant as $1$.

Remove the common determinant twist by putting
$\pi^0=\pi\otimes|\det|_{\mathbb C}^{-m/2}$ and
$W^0(g)=|\det g|_{\mathbb C}^{-m/2}W_{f_{n,\epsilon}}(g)$.
Thus the same quotient $Q(s)$ is obtained by evaluating the corresponding
factors at $s$ for $\pi^0$.

The weight of the minimal $K$-type for the contragredient $(\pi^0)^\vee$ is
$d^\sharp=(-d_n,\ldots,-d_1)$. Denote the Whittaker function of $(\pi^0)^\vee$ via the Jacquet integral for the sections using the dual data by
$W^\vee$ and define
\[
 Q^\vee(s)=\frac{Z(s,W^\vee,\Phi^\vee_{n,\epsilon})}
                    {L(s,(\pi^0)^\vee,\As^+)}.
\]
For the tuple \(d^\sharp=(-d_n,\ldots,-d_1)\), we have
$b^\sharp_{ij}
=\frac{d^\sharp_i-d^\sharp_j}{2}
=\frac{d_{n+1-j}-d_{n+1-i}}2
=b_{n+1-j,n+1-i}$, and the common parity remains $\epsilon$. Thus by definition of \eqref{read:T} and \eqref{read:Schwartz}, we have $T_{n,\epsilon}^\vee=T_{n,\epsilon}$, $\Phi^\vee_{n,\epsilon}=\Phi_{n,\epsilon}$.

\begin{prop}\label{read:analytic-input}
Fix $n\ge2$, $r_i\ge1$, and the common parity $\epsilon$.
Put $\rho=\frac{n^2+n-4}{4}$, then
\begin{enumerate}
\item The Mellin integral $M_{n,\epsilon}(s)$ in \eqref{read:M} and the
Flicker integral $Z(s-m,W_{f_{n,\epsilon}},\Phi_{n,\epsilon})$ in
\eqref{read:Iwasawa} converge absolutely for $\Re s>0$.
\item There exists a constant $C>0$, depending only on
$n$, $r_1,\ldots,r_{n-1}$, and $\epsilon$, such that
\begin{equation}\label{read:right-bound}
 |Q(s)|\le C\left(\frac{1+|s|}{\Re s}\right)^\rho
 \qquad(\Re s\ge\tfrac12).
\end{equation}
\item For $s\in \mathbb{R}$,
$\lim_{s\to+\infty}Q(s)=1.$
\end{enumerate}
The same assertions hold for $Q^\vee$.
\end{prop}
Assertion (1) follows from \cite[Theorem~1,(i)]{beuzart2021archimedean}.
Since $\pi^0$ is tempered,
\cite[Lemma~3.3.1 and the proof of Lemma~3.3.2]{beuzart2021archimedean}
give absolute convergence of $Z(s,W^0,\Phi_{n,\epsilon})$ for
$\Re s>0$. Removing the determinant twist gives
\[
 Z(s-m,W_{f_{n,\epsilon}},\Phi_{n,\epsilon})
 =Z(s,W^0,\Phi_{n,\epsilon}),
\]
with the same absolute convergence. To deduce convergence of
$M_{n,\epsilon}$, put $f=f_{n,\epsilon}$,
$\Phi=\Phi_{n,\epsilon}$, $\sigma=\Re s>0$, and
$a(x)=\operatorname{diag}(x_1\cdots x_{n-1},\ldots,x_{n-1},1)$.
Lemma~\ref{read:Schur}, the triangle inequality, and
\eqref{read:Iwasawa} give
\[
B_\sigma\int_{(0,\infty)^{n-1}}
|F_f(a(x))|\prod_{i<n}x_i^{i\sigma}\frac{dx_i}{x_i}\le
\int_{N_n(\mathbb R)\backslash G_n(\mathbb R)}
|W_f(g)\Phi(e_ng)|\,|\det g|^{\sigma-m}\,dg<\infty,
\]
where
\[
 B_\sigma=|A_{n,\epsilon}|c^{n\epsilon/2}
 \int_0^\infty t^{n(\sigma+\epsilon)}e^{-\pi ct^2}\frac{dt}{t}
 \in(0,\infty).
\]
Thus $M_{n,\epsilon}(s)$ converges absolutely for $\Re s>0$.
The same argument applies to the dual data.

Section~\ref{read:analytic} proves assertions (2) and (3) from the
recurrence relation \eqref{read:recursion}. The proof uses the evaluation of the integral $I_{1,n}(s)$, obtained by
setting $P=1$ in \eqref{pub:network-integral}, and the value of $P_\epsilon$ at the unique
minimum of $\mathcal F=\mathcal V-\log\mathcal A$. We point out that the main reason we use this analytic input
together with the Asai functional equation to prove
Theorem~\ref{read:main} is that the Flicker integral becomes nested
under induction via the Godement section, so the lower-rank
Asai $L$-functions cannot be separated from it directly.
\begin{thm}\label{read:main}
For the section \eqref{read:lowest-polynomials} and the
Schwartz function \eqref{read:Schwartz},
\begin{equation}\label{read:Asai-equality}
 Z(s,W_{f_{n,\epsilon}},\Phi_{n,\epsilon})=L(s,\pi,\As^+).
\end{equation}
The equality holds by absolute convergence for $\Re(s)>-m$
and by meromorphic continuation elsewhere.
\end{thm}
\begin{proof}By Proposition~\ref{read:analytic-input}(1), the Flicker
integral $Z(s,W_{f_{n,\epsilon}},\Phi_{n,\epsilon})$ is absolutely
convergent for $\Re s>-m$. It remains to prove that the integral
equals the Asai $L$-factor.

By \cite[Theorem 3.5.2]{beuzart2021archimedean}, the quotients $Q$ and $Q^\vee$, obtained by dividing the
Flicker integrals by their corresponding Asai $L$-factors,
extend holomorphically to all of $\mathbb C$. Put $\Phi=\Phi_{n,\epsilon}$ and
$\widetilde W^0(g)=W^0(\dot w_ng^{-t})$. Recall that
 $\psi(z)=\psi_{\mathbb R}
 \bigl(\operatorname{Tr}_{\mathbb C/\mathbb R}(z_0z)\bigr)$, where $z_0=-i/2$. Put $N=\binom n2=\frac{n(n-1)}{2}$. The factor contributed by $|z_0|_{\mathbb C}$ in the
functional equation of the Asai L-function is $
 |z_0|_{\mathbb C}^{N(s-1/2)}
 =(1/4)^{N(s-1/2)}=2^{N-2Ns}$.
Thus \cite[Theorem~3.5.1]{beuzart2021archimedean} gives
\begin{equation}\label{read:FE}
 \frac{Z(1-s,\widetilde W^0,\widehat\Phi)}
      {L(1-s,(\pi^0)^\vee,\As^+)}
 =\zeta_\pi\,2^{N-2Ns}Q(s),
\end{equation}
where
\[
 \zeta_\pi=
 \omega_{\pi^0}(-i/2)^{n-1}
 \lambda_{\mathbb C/\mathbb R}(\psi_{\mathbb R})^{-N}
 \varepsilon_{\mathrm{Art}}
 (1/2,\As^+(\phi_{\pi^0}),\psi_{\mathbb R}).
\]
Here $\omega_{\pi^0}$ is the central character of $\pi^0$, $\varepsilon_{\mathrm{Art}}
 (1/2,\As^+(\phi_{\pi^0}),\psi_{\mathbb R})$ is the Artin $\epsilon$-factor. 
$\lambda_{\mathbb C/\mathbb R}(\psi_{\mathbb R})
 =\varepsilon_{\mathrm{Art}}(1/2,\operatorname{sgn},\psi_{\mathbb R})=i$ is the Langlands $\lambda$-function.
The Langlands parameter of the Asai-transfer of $\pi^0$ decomposes as
\[
 \As^+(\phi_{\pi^0})\simeq
 (\operatorname{sgn}^{\epsilon})^{\oplus n}
 \oplus
 \bigoplus_{i<j}
 \operatorname{Ind}_{W_{\mathbb C}}^{W_{\mathbb R}}\chi_{d_i-d_j}.
\]
For $\psi_{\mathbb R}(x)=e^{2\pi ix}$, the epsilon factors
of $\operatorname{sgn}^{\epsilon}$ and $\operatorname{Ind}_{W_{\mathbb C}}^{W_{\mathbb R}}\chi_{d_i-d_j}$ are
$i^\epsilon$ and $i^{d_i-d_j+1}$, respectively
\cite[Section~2.4]{adrianTakeda2024local}.
They are independent of $s$. By
additivity of Artin $\varepsilon$-factors, with $\Omega=\sum_{i<j}b_{ij}$, we obtain
\[
 \varepsilon_{\mathrm{Art}}
 (1/2,\As^+(\phi_{\pi^0}),\psi_{\mathbb R})
 =i^{n\epsilon+\sum_{i<j}(d_i-d_j+1)}
 =i^{n\epsilon+2\Omega+N}.
\]
Moreover, $\omega_{\pi^0}(-i/2)=(-i)^{\sum_i d_i}$, $
 \lambda_{\mathbb C/\mathbb R}(\psi_{\mathbb R})=i$. 
Substituting these values, we obtain
\[
 \zeta_\pi=(-i)^{(n-1)\sum_i d_i}\,
   i^{-N}\,i^{n\epsilon+2\Omega+N}
 =i^{n\epsilon}(-1)^{\sum_i(i-1)d_i}
 =i^{n\epsilon}(-1)^{N\epsilon}
 =i^{n^2\epsilon}.
\]
The second equality uses
$2\Omega=\sum_i(n+1-2i)d_i$, and the third uses
$d_i\equiv\epsilon\pmod2$.

For a homogeneous harmonic polynomial $H$ of degree $\ell$,
the identities
\[
 \widehat{e^{-\pi c\|x\|^2}}(y)=c^{-n/2}e^{-\pi\|y\|^2/c},
 \qquad
 H(\partial_y)e^{-\pi\|y\|^2/c}
 =(-2\pi/c)^\ell H(y)e^{-\pi\|y\|^2/c}
\]
give, with our positive-sign Fourier kernel,
\begin{equation}\label{read:Fourier-harmonic}
 \widehat{H(\sqrt c\,x)e^{-\pi c\|x\|^2}}(y)
 =c^{-n/2}i^\ell H(y/\sqrt c)e^{-\pi\|y\|^2/c}.
\end{equation}

Taking $\ell=n\epsilon$, we obtain
\[
 \widehat\Phi_{A,c}
 =c^{-n/2}i^{n\epsilon}\Phi_{A,1/c},
 \qquad
 \Phi_{A,1/c}(x)=\Phi_{A,c}(x/c).
\]
Since $\pi^0$ has trivial central character on
$\mathbb R_{>0}$, one has
$\widetilde W^0(cg)=\widetilde W^0(g)$ for $c>0$.
The substitution $g=cg'$ preserves the quotient measure
and contributes $c^{n(1-s)}$ from the determinant.
Consequently,
\begin{equation}\label{read:FE-cancellation}
 Z(1-s,\widetilde W^0,\widehat\Phi_{A,c})
 =i^{n\epsilon}c^{n(1/2-s)}
   Z(1-s,\widetilde W^0,\Phi_{A,c})=i^{n\epsilon}2^{N-2Ns}
   Z(1-s,\widetilde W^0,\Phi_{A,c}),
\end{equation}
where the last equality uses $c=2^{n-1}$ and $2N=n(n-1)$.

We next compare $\widetilde W^0$ with $W^\vee$. The map $
 \mathcal T(W)(g)=W(\dot w_ng^{-t})
$
takes the $\psi$-Whittaker model of $\pi^0$ to the
$\psi^{-1}$-Whittaker model of $(\pi^0)^\vee$ and satisfies
$ \mathcal T R(k)=R(\bar k)\mathcal T,
 \ k\in U(n)$. It therefore sends the minimal $U(n)$-type to the dual
minimal $U(n)$-type and commutes with the action of $O(n)$. To restore the character $\psi$ from $\psi^{-1}$, put 
\[
 D_-=\operatorname{diag}((-1)^{n-1},\ldots,-1,1),
 \qquad W^+(g)=\widetilde W^0(D_-g).
\]
We have
$\psi(D_-uD_-^{-1})=\psi^{-1}(u)$ for $u\in N_n(\mathbb C)$.
Thus the generic character for $W^+$ is $\psi$.
Left translation by $D_-$ is invertible and commutes with the
right $K$-action. Thus $W^+$ is also a nonzero $O(n-1)$-fixed vector
of harmonic degree $n\epsilon$ in the dual minimal $K$-type.
By Theorem~\ref{thm:structure}, the $O(n-1)$-fixed line in this
harmonic component is one-dimensional. Both $W^+$ and $W^\vee$
belong to this line, so $W^+=a_\pi W^\vee$
for a nonzero scalar $a_\pi$, independent of $s$. As $D_-$ normalizes $N_n(\mathbb R)$,
and $d(D_- uD_-^{-1})=du$ on $N_n(\mathbb{C})$, we see that $g\mapsto D_-g$ preserves the quotient measure.
It also preserves $|\det g|$ and $e_ng$, since
$|\det D_-|=1$ and $e_nD_-=e_n$.
Using also $\Phi^\vee_{n,\epsilon}=\Phi_{n,\epsilon}
=\Phi_{A,c}$, we obtain
\[
 Z(1-s,\widetilde W^0,\Phi_{A,c})
 =Z(1-s,W^+,\Phi_{A,c})
 =a_\pi Z(1-s,W^\vee,\Phi^\vee_{n,\epsilon}).
\]
Substituting this equality into
\eqref{read:FE-cancellation}, dividing by
$L(1-s,(\pi^0)^\vee,\As^+)$, and comparing with
\eqref{read:FE}, we obtain
\[
 \zeta_\pi\,2^{N-2Ns}Q(s)
 =i^{n\epsilon}2^{N-2Ns}a_\pi Q^\vee(1-s).
\]
Cancelling the nonzero common factor yields
\begin{equation}\label{read:dual-Q}
 Q(s)=c_\pi Q^\vee(1-s),\qquad
 c_\pi=\frac{i^{n\epsilon}a_\pi}{\zeta_\pi}\ne0.
\end{equation}

For $\Re s\ge1/2$, Proposition~\ref{read:analytic-input}(2) gives
$|Q(s)|\le C(1+|s|)^\rho$. For $\Re s\le1/2$, apply the same
bound to $Q^\vee(1-s)$ and use \eqref{read:dual-Q}. Thus the bound
holds on $\mathbb C$. Since $Q$ is entire, Cauchy's estimates give
\[
 |Q^{(j)}(0)|\le j!C\frac{(1+R)^\rho}{R^j}\longrightarrow0
 \qquad(R\to\infty,\ j>\rho).
\]
Hence $Q$ is a polynomial. For real $s\ge1/2$, Proposition~\ref{read:analytic-input} (2)
gives the sharper bound
\[
 |Q(s)|\le C_1\left(\frac{1+s}{s}\right)^\rho
 =C_1(1+s^{-1})^\rho\le3^\rho C_1.
\]
Thus $Q(s)$ is bounded as $s\to+\infty$.
If it had degree $d\ge1$ and leading coefficient $a_d\ne0$,
then $|Q(s)|\sim|a_d|s^d$ as $s\to+\infty$, contradiction. Hence $Q$ is a constant.
Proposition~\ref{read:analytic-input}(3) identifies the constant as $1$.
Evaluating \eqref{read:Q} at $s+m$ proves
$Z(s,W_{f_{n,\epsilon}},\Phi_{n,\epsilon})=L(s,\pi,\As^+)$.
\end{proof}

\section{Analytic estimates and the normalization constant}\label{read:analytic}
We prove parts (2) and (3) of Proposition~\ref{read:analytic-input}
using the recurrence relation \eqref{read:recursion} and
our chosen polynomial realization of test vectors. Iteration expresses $M_{n,\epsilon}(s)$ as
$K_{n,\epsilon}$ times an integral with a homogeneous polynomial
factor. We use this formula to prove the bound and show that
$\lim_{s\to+\infty}Q(s)=1$.
Throughout this section, $n$, $r_i\ge1$, and $\epsilon\in\{0,1\}$ are fixed. Put
$N=\binom n2$, $D=N+n-1$, $ \Omega=\sum_{i<j}b_{ij}=\sum_{j<n}j(n-j)r_j$. Then $\rho=(D-1)/2$ in Proposition~\ref{read:analytic-input}. All constants in this section may depend on the fixed data, but are independent of $s$.

\subsection{Iteration of the recurrence}
\label{pub:analytic-convergence}
For $1\le i\le k\le n$, introduce the variables $a_i^{(k)}>0$, with
$a_n^{(n)}=1$, and write $x_i^{(k)}=\log a_i^{(k)}$.
For each $k$, there are $k$ coordinates $x_i^{(k)}$, $1\le i \le k$.
Since $x_n^{(n)}=0$ is fixed, the remaining
$\sum_{k=1}^n k-1=\binom n2+n-1=D$ coordinates form
$x\in\mathbb R^D$, with Lebesgue measure $dx$. For $1\le k<n$, let $\mathcal B_k$
be the $k\times(k+1)$ matrix whose only nonzero entries are
\begin{equation}\label{pub:matrix-entries}
 D_i^{(k)}=(\mathcal B_k)_{ii}=\frac{a_i^{(k+1)}}{a_i^{(k)}},
 \qquad
 E_i^{(k)}=(\mathcal B_k)_{i,i+1}
                    =\frac{a_i^{(k)}}{a_{i+1}^{(k+1)}}.
\end{equation}
Thus $\mathcal B_k$ is the matrix in \eqref{read:bidiagonal}
with $a_i$ replaced by $a_i^{(k+1)}$ and $t_i$ by $a_i^{(k)}$. Set
\begin{equation}\label{pub:network-definitions}
 X_j=\mathcal B_j\cdots\mathcal B_{n-1},\quad X_n=I_n,\qquad
 \mathcal V=\sum_{k<n}\|\mathcal B_k\|^2,\qquad
 \mathcal A=\prod_{i<n}a_i^{(n)}.
\end{equation}
Here $\|B\|^2=\sum_{i,j}B_{ij}^2=\operatorname{tr}(BB^t)$.
Let $P$ be a polynomial in the entries of the matrices
$\mathcal B_k$. We denote by $P(x)$ the function of the logarithmic
coordinates obtained by substituting
\[
 D_i^{(k)}=e^{x_i^{(k+1)}-x_i^{(k)}},\qquad
 E_i^{(k)}=e^{x_i^{(k)}-x_{i+1}^{(k+1)}},
\]
with $x_n^{(n)}=0$, as in \eqref{pub:matrix-entries}. Define
\begin{equation}\label{pub:network-integral}
 I_{P,n}(s)=\int_{\mathbb R^D}
             P(x)e^{-\pi\mathcal V(x)}\mathcal A(x)^s\,dx.
\end{equation}

Beuzart--Plessis's convergence result \cite{beuzart2021archimedean}
already gives convergence of the Flicker integral and hence of
$M_{n,\epsilon}$, as explained after
Proposition~\ref{read:analytic-input}. It does not directly imply
absolute convergence of \eqref{pub:network-integral}: taking the
absolute value before the internal integrations is a stronger
requirement. The following lemma supplies this control for arbitrary
$P$, together with the estimates needed for Fubini's theorem and
differentiation under the integral.

\begin{lem}\label{pub:convergence}
The integral $I_{P,n}(s)$ converges absolutely and is holomorphic for $\Re s>0$. More generally, the same assertions
hold with $\mathcal A^s$ replaced by
$\prod_{i<n}(a_i^{(n)}/a_{i+1}^{(n)})^{s_i}$ when
$\Re s_i>0$, locally uniformly after any number of derivatives
in each $s_i$.
\end{lem}
\begin{proof}
We prove the multivariable assertion. Let $\ell_e$ range over
$\ell^D_{k,i}=x_i^{(k+1)}-x_i^{(k)}$ and
$\ell^E_{k,i}=x_i^{(k)}-x_{i+1}^{(k+1)}$, so
$\mathcal V=\sum_e e^{2\ell_e}$. Their common kernel is zero:
if all $\ell_e(x)$ vanish, all coordinates equal $x_n^{(n)}=0$.
Consequently, the continuous function $S(x):=\sum_e|\ell_e(x)|$ is positive whenever \(x\ne0\). Let $
c:=\min_{\|u\|=1}S(u)>0$. 
For any \(x\ne0\), by linearity we have
\begin{equation}\label{pub:linear-form-norm}\sum_e|\ell_e(x)|=
S(x)=\|x\|\,S\!\left(\frac{x}{\|x\|}\right)
\ge c\|x\|.
\end{equation} The inequality holds trivially when \(x=0\). 
Put $L(x)=\sum_{i<n}(\Re s_i)(x_i^{(n)}-x_{i+1}^{(n)})$.
The absolute value of the integrand \eqref{pub:network-integral} is then
$|P(x)|e^{-\pi\mathcal V(x)+L(x)}$.
For the special form $\mathcal A(x)^s$, take $s_i=is$.

For positive $\sigma_1,\ldots,\sigma_k$, set
$z_i=\frac12\min(\sigma_i,\sigma_{i+1})$ for $i<k$ and $z_0=z_k=0$.
Then
\begin{equation}\label{pub:positive-decomposition}
\sum_{i=1}^k\sigma_i\log(D_i^{(k)}E_i^{(k)})
 =\sum_{i=1}^k(\sigma_i-z_{i-1})\log D_i^{(k)}
     +\sum_{i=1}^k(\sigma_i-z_i)\log E_i^{(k)}+\sum_{i<k}z_i\log(E_i^{(k)}D_{i+1}^{(k)}).
\end{equation}
Notice that both $\sigma_i-z_{i-1}$ and $\sigma_i-z_{i}$ are strictly positive.
Set $\sigma_i=\Re s_i$, starting with rank $k=n-1$, repeat the same decomposition on the last sum, and continuing
to $k=1$, it expresses $L(x)$ the following form
\[
L(x)=\sum_e\beta_e\ell_e(x),\qquad
\beta_e=
\begin{cases}
\sigma_i-z_{i-1},&\ell_e=\ell^D_{k,i},\\
\sigma_i-z_i,&\ell_e=\ell^E_{k,i}.
\end{cases}
\] Thus $\beta_e>0$ for all $e$. Choose $0<b<\pi$ and $0<\delta<\min_e\beta_e$.
For each $e$, the function
$-be^{2u}+\beta_eu+\delta|u|$ tends to $-\infty$
as $u\to\pm\infty$, so it is bounded above.
Apply this bound with $u=\ell_e(x)$ and sum over $e$.
By \eqref{pub:linear-form-norm}, there exist positive constants $C$ and $c'$ such that
\begin{equation}\label{pub:Gaussian-domination}
 e^{-b\mathcal V(x)+L(x)}
 \le C e^{-\delta\sum_e|\ell_e(x)|}
 \le C e^{-c'\|x\|}.
\end{equation}
Each entry of $\mathcal B_k$ is at most $\sqrt{\mathcal V(x)}$.
Thus a polynomial $P$ of degree $d$ satisfies
$|P(x)|\le C_P(1+\mathcal V(x))^{d/2}$.
Since
$(1+\mathcal V)^{d/2}e^{-(\pi-b)\mathcal V}$ is bounded,
the absolute value of the integrand is bounded by $Ce^{-c'\|x\|}$,
which is integrable.
Derivatives with respect to \(s_1,\ldots,s_{n-1}\) introduce only polynomial factors
in $x$, so they are integrable as well.
Our choices make the positive coefficients $\beta_e$
continuous in $(\Re s_i)$; hence these bounds are uniform
on compact subsets of $\{\Re s_i>0\}$. The last assertion of the lemma is a direct consequence of differentiation under the integral.
\end{proof}

As before, let $\Lambda_j$ be the wedge product of the ordered rows of $X_j$.
Define $P_\epsilon$ as in
\eqref{read:lowest-polynomials} using these $\Lambda_j$'s via \eqref{read:Gram} and \eqref{read:B-definition}. That is,
$P_\epsilon(x):=\left.P_{n,\epsilon}\right|_{\Lambda_j=\Lambda_j(x),\,1\le j<n}$, where $P_{n,\epsilon}$ is given by \eqref{read:lowest-polynomials}.
Each entry of $X_j=\mathcal B_j\cdots\mathcal B_{n-1}$ is a sum of products containing one entry from each factor, thus it is
homogeneous of degree $n-j$. Using the standard dual basis $e_1^*,\cdots, e_n^*$ of $V^*$, write
\[
\Lambda_j
=\sum_{i_1<\cdots<i_j}
\det\!\left((X_j)_{a,i_b}\right)_{1\le a,b\le j}
\,e_{i_1}^*\wedge\cdots\wedge e_{i_j}^*,
\] $\det\!\left((X_j)_{a,i_b}\right)$ are $j\times j$ minors of $X_j$,
so they have degree $j(n-j)$.
The factor $\mathcal G_j^{r_j}$ has degree $2r_j$ in $\Lambda_j$.
For $\epsilon=1$, $B_n$ contributes degree $2$ and
$\mathcal G_j^{r_j-1}$ contributes $2(r_j-1)$.
Thus, for both parities,
\begin{equation}\label{pub:network-degree}
 \deg P_\epsilon=2\sum_{j<n}j(n-j)r_j=2\Omega.
\end{equation}
The last equality holds because $r_j$ occurs in $b_{ik}$
exactly when $i\le j<k$, in total there are $j(n-j)$ of them.

\begin{prop}\label{upper:iteration}
For $\Re s>0$, we have
\begin{equation}\label{upper:period-integral}
 M_{n,\epsilon}(s)=K_{n,\epsilon}I_{P_\epsilon,n}(s),
 \quad \text{where} \quad
 K_{n,\epsilon}=
 \frac{(-1)^{N\epsilon}2^N\pi^{N+\Omega}}
      {\prod_{i<j}\Gamma(1+b_{ij})}.
\end{equation}
\end{prop}

\begin{proof}
In the tensor realization of $V_n$, the composite of the maps
$D_j\mathbb S_{\eta^{[j+1]}}(\mathcal B_j)$ contracts each
$\bigwedge^j\mathbb C^n$ factor with the row wedge $\Lambda_j$ of
$X_j=\mathcal B_j\cdots\mathcal B_{n-1}$. Indeed,
$\bigwedge^jX_j$ sends $\xi$ to
$\Lambda_j(\xi)(e_1\wedge\cdots\wedge e_j)$, and $D_j$ removes
the resulting unit determinant factor
$(e_1\wedge\cdots\wedge e_j)^{\otimes2r_j}$. Hence, for $v\in V_n$,
\[
 (D_1\mathbb S_{\eta^{[2]}}(\mathcal B_1))\cdots
 (D_{n-1}\mathbb S_{\eta^{[n]}}(\mathcal B_{n-1}))v
 =\left(\bigotimes_{j<n}\Lambda_j^{\otimes2r_j}\right)(v)=P_v(x).
\]
Since $X_j=\mathcal B_jX_{j+1}$, 
$\operatorname{row}(X_j)\subseteq\operatorname{row}(X_{j+1})$. Hence there exists $X\in\mathrm{GL}_n(\mathbb C)$ such that the wedge of its first $j$ rows equals $\Lambda_j$ for every $1\le j<n$.
By \eqref{cartan:coefficient} and \eqref{read:lowest-polynomials},
$P_{v_{n,\epsilon}}(x)=P_\epsilon(x)$.

The polynomial $P_v$ is linear in $v$ and homogeneous of
degree $2\Omega$, so $
 |P_v(x)|\le C\|v\|\mathcal V(x)^\Omega$ for some $C>0$.
For $\sigma=\Re s>0$, Lemma~\ref{pub:convergence} gives
$I_{\mathcal V^\Omega,n}(\sigma)<\infty$. Its estimate \eqref{pub:Gaussian-domination}, after absorbing the
polynomial factor into the Gaussian, also gives integrability in
the variables $x_i^{(k)}=\log a_i^{(k)}$,
$1\le i\le k<n$, for every fixed
$(a_1^{(n)},\ldots,a_{n-1}^{(n)})$.
Thus iteration of \eqref{read:recursion} and Fubini's theorem
are justified.
The constants in the formula \eqref{read:recursion} multiply to
\[
 \prod_{k=1}^{n-1}(-1)^{k\epsilon}\kappa_{k+1}
 =\frac{(-1)^{N\epsilon}2^N\pi^{N+\Omega}}
        {\prod_{i<j}\Gamma(1+b_{ij})}
 =K_{n,\epsilon}.
\]
Setting $y_i=a_i^{(n)}/a_{i+1}^{(n)}$, we have
$\prod_{i<n}y_i^{is}=\mathcal A^s$, $
 \prod_{i<n}\frac{dy_i}{y_i}=\prod_{i<n}dx_i^{(n)}$. Hence
\[
 M_{n,\epsilon}(s)
 =K_{n,\epsilon}\int_{\mathbb R^D}
      P_\epsilon(x)e^{-\pi\mathcal V(x)}\mathcal A(x)^s\,dx
 =K_{n,\epsilon}I_{P_\epsilon,n}(s),
\]
which proves \eqref{upper:period-integral}.
\end{proof}

\subsection{A bound by the integral with \texorpdfstring{$P=1$}{P=1}}
Let $\mathbf v=(v_i^k)_{1\le i\le k\le n}$ with $v_i^{(k)}=n+k-2i$,
excluding $(k,i)=(n,n)$, and put
$W=\{y:y_{n-1}^{(n-1)}=0\}$. Then every $x$ has a unique expression
$x=y+t\mathbf v$, with $y\in W$; the Jacobian of the change of variable $x\mapsto (y,t)$ is 1 because
$v_{n-1}^{(n-1)}=1$. We also have $
\mathcal B_k(y+t\mathbf v)=e^t\mathcal B_k(y)$, for $1\le k<n $.
Consequently, if $P$ is homogeneous of degree $2\omega$, then
\begin{equation}\label{pub:common-scale}
 \mathcal V(y+t\mathbf v)=e^{2t}\mathcal V(y),\quad
 \mathcal A(y+t\mathbf v)=e^{2Nt}\mathcal A(y),\quad
 P(y+t\mathbf v)=e^{2\omega t}P(y).
\end{equation}
We will use $\omega=\Omega$ for $P=P_\epsilon$, and $\omega=0$
for $P=1$.

\begin{lem}\label{pub:scale-disintegration}
For a homogeneous polynomial $P$ of degree $2\omega\ge0$
and $\Re s>0$,
\begin{equation}\label{pub:scale-formula}
 I_{P,n}(s)=\frac{\Gamma(Ns+\omega)}{2\pi^{Ns+\omega}}
 \int_W\frac{P(y)\mathcal A(y)^s}
                 {\mathcal V(y)^{Ns+\omega}}\,dy.
\end{equation}
This integral is absolutely convergent. In particular, with
$\sigma=\Re s$,
\begin{equation}\label{pub:scale-bound}
 |I_{P,n}(s)|\le C_P\pi^{-\omega}
 \frac{|\Gamma(Ns+\omega)|}{\Gamma(N\sigma)}I_{1,n}(\sigma).
\end{equation}
\end{lem}
\begin{proof}
Lemma~\ref{pub:convergence} permits the change $x=y+t\mathbf v$
and integration in $t$. With $u=\pi\mathcal V(y)e^{2t}$,
\[
 \int_{\mathbb R}e^{-\pi\mathcal V(y)e^{2t}}e^{2(Ns+\omega)t}\,dt
 =\frac{\Gamma(Ns+\omega)}{2(\pi\mathcal V(y))^{Ns+\omega}},
\]
which proves \eqref{pub:scale-formula}. Applying the same calculation
to the absolute value at $\sigma=\Re s$ gives
\[
 \int_{\mathbb R^D}|P(x)|e^{-\pi\mathcal V(x)}\mathcal A(x)^\sigma\,dx
 =\frac{\Gamma(N\sigma+\omega)}{2\pi^{N\sigma+\omega}}
 \int_W\frac{|P(y)|\mathcal A(y)^\sigma}
 {\mathcal V(y)^{N\sigma+\omega}}\,dy<\infty.
\]
This also proves absolute convergence of the integral over $W$.
Homogeneity gives $|P(y)|\le C_P\mathcal V(y)^\omega$;
comparison with \eqref{pub:scale-formula} for $P=1$ at $\sigma$
now yields \eqref{pub:scale-bound}.
\end{proof}

\subsection{Evaluation of the integral with \texorpdfstring{$P=1$}{P=1}}
\label{pub:spherical-calibration-section}
Let $U_1=1$ and, for $a=\mathrm{diag}(a_1,\cdots, a_{k+1})\in A_{k+1}(\mathbb{R})^+$, define
\begin{equation}\label{pub:spherical-Gaussian-rec}
 U_{k+1}(a)=\int_{(0,\infty)^k}U_k(t)
 e^{-\pi\sum_{i=1}^k((a_i/t_i)^2+(t_i/a_{i+1})^2)}\prod_{i=1}^k\frac{dt_i}{t_i}.
\end{equation}
These functions are positive and satisfy $U_k(ca)=U_k(a)$
for every $c>0$. Put $\mathcal M_1=1$ and
\[
 \mathcal M_k(\mathbf s)=
 \int_{(0,\infty)^{k-1}}U_k(a(y))
          \prod_{i<k}y_i^{s_i}\frac{dy_i}{y_i},\qquad
 a(y)=\operatorname{diag}(y_1\cdots y_{k-1},\ldots,y_{k-1},1).
\]
The integrals converge when $\Re s_i>0$ by
Lemma~\ref{pub:convergence}. 
All contours $(\mathbf z)$ below have the convention of
Section~\ref{read:notations}.

\begin{lem}\label{pub:spherical-recurrence}Set
$g(u)=\frac12\pi^{-u/2}\Gamma(u/2)$.
If $\Re s_i>0$ and
$0<c_i<\min(\Re s_i,\Re s_{i+1})$ for $1\le i<k$, then
\begin{equation}\label{pub:spherical-Mellin-rec}
 \mathcal M_{k+1}(\mathbf s)=
 \int_{(\mathbf z)}\mathcal M_k(\mathbf z)
 \prod_{i=1}^k g(s_i-z_{i-1})g(s_i-z_i)
                         \prod_{i<k}\frac{dz_i}{2\pi i},
 \qquad z_0=z_k=0,\quad\Re z_i=c_i.
\end{equation}
The contour integral converges absolutely.
\end{lem}

\begin{proof}
For $k\ge2$, put $q_i=\log y_i$ and
$f(q)=e^{\mathbf c\cdot q}U_k(a(e^q))$.
Iteration of \eqref{pub:spherical-Gaussian-rec} expresses $f$
as a Gaussian integral. Differentiating the integrand defining $f$ in logarithmic
coordinates introduces only polynomials in the matrix entries, so \eqref{pub:Gaussian-domination} shows
that $f$ and all its derivatives are integrable. Put
$\Delta_q=\sum_{i<k}\partial_{q_i}^2$.
Integration by parts, justified first with a compactly supported
cutoff and then by $L^1$ convergence, gives
\begin{equation}\label{eq:rapid-decay-Mk}
    |\mathcal M_k(\mathbf c+i\boldsymbol\tau)|
 \le \|(1-\Delta_q)^m f\|_1(1+\|\boldsymbol\tau\|^2)^{-m}
 \qquad(m\ge0).
\end{equation}
Next, for \(t=\operatorname{diag}(t_1,\ldots,t_k)\), set \(y_i=t_i/t_{i+1}\), then \(a(y)=t_k^{-1}t\). Thus \(U_k(a(y))=U_k(t)\), together with
Mellin inversion formula, this gives
\begin{equation}\label{eq:Mellin-inversion}
 U_k(t)=\int_{(\mathbf z)}\mathcal M_k(\mathbf z)
 \prod_{i<k}(t_i/t_{i+1})^{-z_i}
 \prod_{i<k}\frac{dz_i}{2\pi i}.
\end{equation}
For $k=1$ the formula holds with no integration.

Insert \eqref{eq:Mellin-inversion} into the definition of
$\mathcal M_{k+1}$ and its Gaussian recurrence. With $a=a(y)$, set
$D_i=a_i/t_i$, $E_i=t_i/a_{i+1}$. Then
$y_i=D_iE_i$, $t_i/t_{i+1}=E_iD_{i+1}$, and the logarithmic
Jacobian is one. With $z_0=z_k=0$, the integral becomes
\[
 \int_{(\mathbf z)}\mathcal M_k(\mathbf z)
 \int_{(0,\infty)^{2k}}
 \prod_{i=1}^k D_i^{s_i-z_{i-1}}E_i^{s_i-z_i}
 e^{-\pi(D_i^2+E_i^2)}
 \frac{dD_i}{D_i}\frac{dE_i}{E_i}
 \prod_{i<k}\frac{dz_i}{2\pi i}.
\]
To justify the interchange, put $c_0=c_k=0$.
Tonelli's theorem and the Gaussian Mellin integral give the absolute
integral
\[
 \prod_{i=1}^k g(\Re s_i-c_{i-1})g(\Re s_i-c_i)
 \int_{\mathbb R^{k-1}}|\mathcal M_k(\mathbf c+i\boldsymbol\tau)|
             \frac{d\boldsymbol\tau}{(2\pi)^{k-1}}<\infty.
\]
Here the contour inequalities give positive real parts for the Mellin exponents,
and the last integral is finite by \eqref{eq:rapid-decay-Mk}.
Applying Fubini's theorem and integrating in $D_i,E_i$ proves
\eqref{pub:spherical-Mellin-rec} and its absolute convergence.
\end{proof}

\begin{lem}\label{pub:Gaussian-convolution}
Suppose $\Re\sigma>0$, $\Re w_i>0$, and $w_0=0$.
On contours $(\mathbf z)$ with
$0<c_i<\min(\Re w_i,\Re w_{i-1}+2\Re\sigma)$,
\begin{equation}\label{pub:Gaussian-convolution-formula}
\int_{(\mathbf z)}\mathcal M_k(\mathbf z)
 \prod_{i<k}g(w_i-z_i)g(w_{i-1}+2\sigma-z_i)
\prod_{i<k}\frac{dz_i}{2\pi i}=2^{-(k-1)}\pi^{w_{k-1}/2-(k-1)\sigma}
 \frac{\Gamma(\sigma)^k}{\Gamma(w_{k-1}/2+\sigma)}\mathcal M_k(\mathbf w).
\end{equation}
Both sides are trivial for $k=1$.
\end{lem}
\begin{proof}
This is the Ishii--Stade
convolution \cite[Proposition~6.1]{chen-cheng-theta} with its parameter tuple $\underline{\mu}=(0,\ldots,0)\in\mathbb C^k$,
after converting its Mellin transform and gamma factors
to our normalization. It is immediate for $k=1$. For $k\ge2$, set
\[
 \mathfrak M_k(\mathbf s)
 =2^{4\binom k2-\sum_i s_i}\mathcal M_k(\mathbf s).
\]
Using $\Gamma_{\mathbb C}(u/2)=2^{2-u/2}g(u)$ in
\eqref{pub:spherical-Mellin-rec}, we have
\[
 \mathfrak M_{k+1}(\mathbf s)
 =\int_{(\mathbf z)}\mathfrak M_k(\mathbf z)
 \prod_{i=1}^k
 \Gamma_{\mathbb C}\!\left(\frac{s_i-z_{i-1}}2\right)
 \Gamma_{\mathbb C}\!\left(\frac{s_i-z_i}2\right)
 \prod_{i<k}\frac{dz_i}{2\pi i},
 \qquad z_0=z_k=0.
\]
In the notation of \cite[Section~5.3]{chen-cheng-theta}, $\mathfrak M_2(s)=\Gamma_{\mathbb C}(s/2)^2=16(Mf_{(0,0)})(s)$.
Since $\mathfrak M_k$ and $Mf_{\mathbf0}$ satisfy the same
recurrence \cite[(5.3)]{chen-cheng-theta}, by induction,
$\mathfrak M_k=16Mf_{\mathbf0}$ for every $k\ge2$,
where $\mathbf0\in\mathbb C^k$ means that all inducing
parameters are zero. Applying \cite[Proposition~6.1]{chen-cheng-theta}
with this tuple, and cancelling the common factor $16$, yields
\[
 \int_{(\mathbf z)}\mathfrak M_k(\mathbf z)
 \prod_{i<k}
 \Gamma_{\mathbb C}\!\left(\frac{w_i-z_i}2\right)
 \Gamma_{\mathbb C}\!\left(\frac{w_{i-1}+2\sigma-z_i}2\right)
 \prod_{i<k}\frac{dz_i}{2\pi i}=
 2^{2(k-1)}
 \frac{\Gamma_{\mathbb C}(\sigma)^k}
      {\Gamma_{\mathbb C}(w_{k-1}/2+\sigma)}
 \mathfrak M_k(\mathbf w).
\]
Substitute the definitions of $\mathfrak M_k$ and
$\Gamma_{\mathbb C}(u/2)$.
The powers involving $\sum_i z_i$ cancel. Since $w_0=0$,
$\sum_{i<k}(w_i+w_{i-1})=2\sum_{i<k}w_i-w_{k-1}$.
Thus the factor before $\mathcal M_k(\mathbf w)$ is
\[
 2^{-2(k-1)-w_{k-1}/2+(k-1)\sigma}
 \frac{\Gamma_{\mathbb C}(\sigma)^k}
      {\Gamma_{\mathbb C}(w_{k-1}/2+\sigma)}=
 2^{-(k-1)}\pi^{w_{k-1}/2-(k-1)\sigma}
 \frac{\Gamma(\sigma)^k}{\Gamma(w_{k-1}/2+\sigma)}.
\]
The contour $(\mathbf z)$ provides
$\Re(w_i-z_i)>0$ and
$\Re(w_{i-1}+2\sigma-z_i)>0$.
Moreover,
$|\mathcal M_k(\mathbf c+i\mathbf t)|\le\mathcal M_k(\mathbf c)$.
Stirling's formula gives exponential decay of the two gamma
factors in each $t_i$, proving the absolute convergence.
\end{proof}

\begin{prop}\label{pub:scalar-calibration}Set $I_{1,1}=1$.
For $\Re s>0$, and $n\ge 2$, we have
\begin{equation}\label{pub:I1-value}
 I_{1,n}(s)=2^{-D}\pi^{-Ns}G_0(s),\qquad
 G_0(s)=\frac{\Gamma(s/2)^n\Gamma(s)^N}{\Gamma(ns/2)}.
\end{equation} 
\end{prop}
\begin{proof}
By definition,
$I_{1,n}(s)=\mathcal M_n(s,2s,\ldots,(n-1)s)$.
In \eqref{pub:spherical-Mellin-rec}, take $s_i=is$.
The factors independent of the contour variables are
$g(s)g((n-1)s)$. Apply
Lemma~\ref{pub:Gaussian-convolution} to the remaining integral
in rank $n-1$, with $w_i=is$ and $\sigma=s$. This gives
\begin{equation}\label{pub:I1-rec}
 I_{1,n}(s)=2^{-n}\pi^{-(n-1)s}
 \frac{\Gamma(s/2)\Gamma((n-1)s/2)}{\Gamma(ns/2)}
                    \Gamma(s)^{n-1}I_{1,n-1}(s).
\end{equation}
For $n=2$, the same identity is $I_{1,2}(s)=g(s)^2$.
Iterating \eqref{pub:I1-rec} from $I_{1,1}=1$, the gamma ratios
cancel:
\[
 \prod_{k=2}^n
 \frac{\Gamma((k-1)s/2)}{\Gamma(ks/2)}
 =\frac{\Gamma(s/2)}{\Gamma(ns/2)}.
\]
Since $\sum_{k=2}^n k=D$ and $\sum_{k=2}^n(k-1)=N$, we obtain
\[
 I_{1,n}(s)
 =2^{-D}\pi^{-Ns}
 \frac{\Gamma(s/2)^n\Gamma(s)^N}{\Gamma(ns/2)},
\]
which proves \eqref{pub:I1-value}. The contours $(\mathbf z)$ with
$c_i=i\Re s/2$ satisfy all the required inequalities.
\end{proof}

\subsection{Proof of
Proposition~\ref{read:analytic-input} (2)}\label{pub:right-bound-section} Denote $\mathbf r=(r_1,\cdots, r_{n-1})$, with $r_i=\frac{d_i-d_{i+1}}{2}$ as in \eqref{eq:ri}. $b_{ij}=\frac{d_i-d_j}{2}=\sum_{k=i}^{j-1}r_k$ as in \eqref{eq:bij-N-s-alpha-l}.
Put
\[
 G_{\mathbf r,\epsilon}(s)=
 \frac{\Gamma((s+\epsilon)/2)^n}{\Gamma(n(s+\epsilon)/2)}\prod_{i<j}\Gamma(s+b_{ij}).
\]
By equations \eqref{read:T} and \eqref{upper:period-integral},
\[
 \frac{T_{n,\epsilon}(s)}{K_{n,\epsilon}}=2^{-D}\pi^{-Ns-\Omega}G_{\mathbf r,\epsilon}(s).
\]
Therefore \eqref{pub:scale-bound}, applied with
$P=P_\epsilon$ and $\omega=\Omega$, and
\eqref{pub:I1-value} imply
\begin{equation}\label{pub:Q-bound}
 |Q(s)|\le C_P
 \frac{|\Gamma(Ns+\Omega)|}{|G_{\mathbf r,\epsilon}(s)|}
 \frac{G_0(\sigma)}{\Gamma(N\sigma)},\qquad \sigma=\Re s>0.
\end{equation}
Put $a_0=N^Nn^{n/2}$, $\sigma=\Re s$, and recall that $\rho=(D-1)/2$. 
Define all complex powers of $s$ using
$\mathrm{Log}\,s=\log|s|+i\arg s$, where $-\pi/2<\arg s<\pi/2$. By
Stirling's formula, we obtain the following uniform estimates
as $|s|\to\infty$ in $\Re s\ge\tfrac12$: 
\[
 \begin{aligned}
 G_{\mathbf r,\epsilon}(s)
 &=C_1n^{-ns/2}e^{-Ns}s^{Ns+\Omega-D/2}
   \bigl(1+O(|s|^{-1})\bigr),\\
 \Gamma(Ns+\Omega)
 &=C_2N^{Ns}e^{-Ns}s^{Ns+\Omega-1/2}
   \bigl(1+O(|s|^{-1})\bigr),
 \end{aligned}
\]
where $C_1,C_2>0$ are independent of $s$. Indeed, the quotient
$\Gamma((s+\epsilon)/2)^n/\Gamma(n(s+\epsilon)/2)$ in $G_{\mathbf r,\epsilon}(s)$
contributes $n^{-ns/2}s^{-(n-1)/2}$, while the remaining
$N$ gamma factors contribute $e^{-Ns}s^{Ns+\Omega-N/2}$.
Here we used $\sum_{i<j}b_{ij}=\Omega$ and $D=N+n-1$.
Dividing the above estimations of $\Gamma(Ns+\Omega)$ by that of $G_{\mathbf r, \epsilon}(s)$, we obtain
\[
 \frac{\Gamma(Ns+\Omega)}{G_{\mathbf r,\epsilon}(s)}
 =C_0\bigl(N^Nn^{n/2}\bigr)^s s^{(D-1)/2}
   \bigl(1+O(|s|^{-1})\bigr)
 =C_0a_0^s s^\rho\bigl(1+O(|s|^{-1})\bigr),
 \qquad C_0=C_2/C_1>0.
\]
Consequently,
\begin{equation}\label{pub:uniform-Stirling}
 \frac{|\Gamma(Ns+\Omega)|}{|G_{\mathbf r,\epsilon}(s)|}
 \le C a_0^\sigma(1+|s|)^\rho,\qquad
 \frac{G_0(\sigma)}{\Gamma(N\sigma)}
 \le C a_0^{-\sigma}\sigma^{-\rho}
 \quad(\sigma\ge\tfrac12),
\end{equation} where the second inequality is obtained by applying the same calculation with $\epsilon=b_{ij}=0$
and taking the reciprocal on the positive real axis.
These bounds initially hold for sufficiently large $|s|$ and $\sigma$,
respectively. All gamma factors are finite and nonzero in
the stated domains, by continuity, we may enlarge $C$ so that both bounds hold
for all $\Re s\ge\tfrac12$ and all $\sigma\ge\tfrac12$,
respectively.
Substituting \eqref{pub:uniform-Stirling} into \eqref{pub:Q-bound} proves
Proposition~\ref{read:analytic-input} (2).

\begin{remark}The estimate with exponent $\rho$ relies on the precise structure
of the recurrence \eqref{read:recursion}. Its iteration determines
both the factor
$K_{n,\epsilon}=\prod_{k=1}^{n-1}(-1)^{k\epsilon}\kappa_{k+1}$
and, for our chosen test vector, the homogeneous polynomial of degree $2\Omega$.
This product fixes the exact normalization of
$T_{n,\epsilon}/K_{n,\epsilon}$ in \eqref{pub:Q-bound}.
Since $K_{n,\epsilon}$ is independent of $s$, its value affects
only the overall constant in the bound; its exact value is
essential for the limit in part (3). The degree $2\Omega$
produces $\Gamma(Ns+\Omega)$ through integration in the common
scaling variable. Its shift $\Omega$ matches the total shift
$\sum_{i<j}b_{ij}$ in $G_{\mathbf r,\epsilon}$, so the
weight-dependent powers cancel and leave precisely
$\rho=(D-1)/2$. Thus this exponent follows from \eqref{read:recursion} and the
homogeneity of the polynomial associated with our chosen test vector.
\end{remark}

\subsection{The minimum and the polynomial value}\label{pub:saddle-section}
Set $\mathcal F=\mathcal V-\log\mathcal A$.

\begin{prop}\label{pub:saddle-main}
The function $\mathcal F$ has a unique minimum $x_*$ and
satisfies $\mathcal F(x)\ge c\|x\|-C$ for some constants
$c,C>0$.
At $x_*$,
\begin{equation}\label{pub:saddle-values}
 \mathcal V=N,\qquad \mathcal A=n^{-n/2},\qquad
 \det(X_jX_j^t)=1\ (1\le j<n),\qquad
 P_\epsilon=n^{-n\epsilon/2}.
\end{equation}
\end{prop}
\begin{proof}
By \eqref{pub:positive-decomposition},
$\log\mathcal A=\sum_e\beta_e\ell_e$ with $\beta_e>0$.
Apply \(\eqref{pub:Gaussian-domination}\) with \(b=1\) and \(L=\log\mathcal A\), we get
$
e^{-\mathcal F(x)}
=e^{-\mathcal V(x)+\log\mathcal A(x)}
\le C_0e^{-c\|x\|}$. Taking logarithms and put $C=\log C_0$, we get
$\mathcal F(x)\ge c\|x\|-C$, so a minimum exists. We may enlarge $C_0$ so that $C>0$.
For $h\ne0$,
\[
 \left.\frac{d^2}{dt^2}\mathcal F(x+th)\right|_{t=0}
             =4\sum_e e^{2\ell_e(x)}\ell_e(h)^2>0
\]
by \eqref{pub:linear-form-norm}. Thus the minimum is unique.
Differentiating in the direction $\mathbf v$ of
\eqref{pub:common-scale}, we obtain $\mathcal V(x_*)=N$.

All matrices in the rest of the proof are evaluated at $x_*$. Set
$D_i^{(k)}=E_i^{(k)}=0$ for indices outside $1\le i\le k<n$.
Differentiating $\mathcal F$ with respect to each
$x_i^{(k)}$, $(k,i)\ne(n,n)$, and setting the derivatives
equal to zero at $x_*$, we obtain 
\begin{equation}\label{short:critical-equations}
\begin{aligned}
 (D_i^{(k)})^2-(E_i^{(k)})^2
 &=(D_i^{(k-1)})^2-(E_{i-1}^{(k-1)})^2
 &&(1\le i\le k<n),\\
 (D_i^{(n-1)})^2-(E_{i-1}^{(n-1)})^2
 &=\tfrac12
 &&(1\le i<n).
\end{aligned}
\end{equation}
Recall that
$\mathcal A
 =\prod_{i<n}\left(\frac{a_i^{(n)}}{a_{i+1}^{(n)}}\right)^i
 =\prod_{i<n}\bigl(D_i^{(n-1)}E_i^{(n-1)}\bigr)^i$. To compute its value,
we express it as a determinant built from powers of a tridiagonal matrix.
Take $L_1=0$ and define for $1\le k<n$,
\[
 \begin{aligned}
 (L_{k+1})_{ii}
   &=(D_i^{(k)})^2-(E_{i-1}^{(k)})^2
       &&(1\le i\le k+1),\\
 (L_{k+1})_{i,i+1}
   &=D_i^{(k)}E_i^{(k)},&
 (L_{k+1})_{i+1,i}
   &=-D_i^{(k)}E_i^{(k)}
       &&(1\le i\le k),
 \end{aligned}
\]
with all other entries zero. Let $e_n=(0,\cdots,0, 1)$ be the row vector as before. Thus
$(L_n)_{i+1,i}=-a_i^{(n)}/a_{i+1}^{(n)}$, and
\[
 (e_nL_n^r)_i=0\quad(i<n-r),\qquad
 (e_nL_n^r)_{n-r}
 =(-1)^r\prod_{i=n-r}^{n-1}
              \frac{a_i^{(n)}}{a_{i+1}^{(n)}}.
\]
Expansion along the anti-diagonal entries, we obtain
\[
 \det\begin{pmatrix}
 e_n\\ e_nL_n\\ \vdots\\ e_nL_n^{n-1}
 \end{pmatrix}
 =(-1)^N\prod_{r=0}^{n-1}
   \left[(-1)^r\prod_{i=n-r}^{n-1}
              \frac{a_i^{(n)}}{a_{i+1}^{(n)}}\right]=\prod_{i<n}
   \left(\frac{a_i^{(n)}}{a_{i+1}^{(n)}}\right)^i
 =\mathcal A.
\]
Here each ratio occurs exactly $i$ times, and
$\sum_{r=0}^{n-1}r=N$ cancels the signs.

We evaluate this determinant by putting $L_n$ into an
explicit triangular form. Set
$S_k=\operatorname{diag}(1,-1,\ldots,(-1)^{k-1})$.
The equations \eqref{short:critical-equations} and $E_i^{(k)}D_{i+1}^{(k)}
 =D_i^{(k-1)}E_i^{(k-1)}$ give
\begin{equation}\label{eq:rectangular}
     L_{k+1}=S_{k+1}\mathcal B_k^tS_k\mathcal B_k,\qquad
 \mathcal B_kS_{k+1}\mathcal B_k^tS_k=L_k;
\end{equation}
Consequently, by induction from $L_1=0$, we obtain
\begin{equation}\label{short:matrix-identities}
 \mathcal B_kL_{k+1}=L_k\mathcal B_k,\qquad
 \det(zI_{k+1}-L_{k+1})
 =z\det(zI_k-L_k)=z^{k+1}.
\end{equation}
Thus every $L_k$ is nilpotent.
Its off-diagonal entries
are skew-symmetric, so \eqref{short:critical-equations}
and $\operatorname{tr}L_n=0$ imply
\begin{equation}\label{short:symmetric-part}
 L_n+L_n^t=I_n-ne_n^te_n.
\end{equation}

Since $\operatorname{rank}X_j=j$ and
$\operatorname{row}(X_j)\subset\operatorname{row}(X_{j+1})$,
choose the rows of $U$ successively as unit vectors in
$\operatorname{row}(X_j)$ orthogonal to the preceding rows.
Then $U$ is orthogonal. Moreover,
$X_jL_n=L_jX_j$, $L_n$ is nilpotent, so $T=UL_nU^t$ is lower triangular and nilpotent, thus
$T_{ii}=0$. Hence
\[
 T+T^t=I_n-n(Ue_n^t)(e_nU^t)
 \quad\Longrightarrow\quad U_{in}^2=1/n.
\]
Multiplying the $i$-th row by $-1$ whenever $U_{in}<0$
preserves orthogonality and the span of the first $j$ rows
for every $j$, so we may choose $U_{in}=n^{-1/2}$ for all $i$.
The off-diagonal entries of $T+T^t$ are then $-1$, thus $T_{ij}=-1$ if $i>j$, and $0$ otherwise.

For $0\le r<n$, we have
$e_nL_n^r=(e_nU^t)T^rU$ and
\[
 \bigl((e_nU^t)T^r\bigr)_i=0\quad(i>n-r),\qquad
 \bigl((e_nU^t)T^r\bigr)_{n-r}=(-1)^rn^{-1/2}.
\]
Thus the determinant expression for $\mathcal A$ becomes
\[
 \mathcal A
 =\det\begin{pmatrix}
 e_n\\ e_nL_n\\ \vdots\\ e_nL_n^{n-1}
 \end{pmatrix}
 =\det\begin{pmatrix}
 e_nU^t\\ (e_nU^t)T\\ \vdots\\ (e_nU^t)T^{n-1}
 \end{pmatrix}\det U=(-1)^N\prod_{r=0}^{n-1}
       \bigl((-1)^rn^{-1/2}\bigr)\det U
 =n^{-n/2}\det U.
\]
Since $\mathcal A>0$ and $\det U=\pm1$, $\det U=1$ and $\mathcal A=n^{-n/2}$.

To evaluate $P_\epsilon(x_*)$, we first compute
$\|\Lambda_j(x_*)\|^2$ for $1\le j<n$.
Iterating the matrix identities gives
\begin{equation}\label{short:power-factorization}
 L_n^{n-j}=S_nX_j^tS_jX_j.
\end{equation}

By Cauchy-Binet formula and \eqref{short:power-factorization},
\[\det(X_jX_j^t)^2=\left(\sum_{|I|=j}\det((X_j)_{:,I})^2\right)^2=\sum_{|I|=|J|=j}\det\bigl((L_n^{n-j})_{I,J}\bigr)^2,
\]
where $(X_j)_{:,I}$ denotes the matrix obtained by selecting the columns of $X_j$ indexed by $I$ and retaining all rows.
The second equality uses that the diagonal entries of
$S_n$ and $S_j$ are $\pm1$.
Since $\bigwedge^jU$ is orthogonal, this sum is unchanged
when $L_n^{n-j}$ is replaced by $T^{n-j}=UL_n^{n-j}U^t$.
The only nonzero $j\times j$ minor of $T^{n-j}$ is given by its
last $j$ rows and first $j$ columns; this block is lower
triangular with diagonal $(-1)^{n-j}$.
Thus $\det(X_jX_j^t)^2=1$, and positive definiteness gives
 $\|\Lambda_j\|^2=\det(X_jX_j^t)=1$.
The first $j$ rows of $U$ and the rows of $X_j$ span the
same space. Their wedges have norm one, so they differ
only by a sign. These signs do not change $P_\epsilon$,
since its degree in each $\Lambda_j$ is $2r_j$.
Finally, since the rows of $U$ are orthogonal, $\mathcal G_j(U)=1$, \eqref{read:B-orthogonal} gives
\[
 P_\epsilon(x_*)=P_{n,\epsilon}(U)
 =\left(\det U\prod_{i=1}^nU_{in}\right)^\epsilon
 =n^{-n\epsilon/2},
\]
where $\det U=1$ and $U_{in}=n^{-1/2}$.
\end{proof}

\subsection{Proof of Proposition~\ref{read:analytic-input} (3)}\label{pub:Laplace-section}
The conclusion $\lim_{s\to+\infty}Q(s)=1$ depends on the precise
constants and polynomial evaluations established above. Here the
cohomological hypothesis supplies the common inducing exponent and
the strictly decreasing weights of common parity used in
\eqref{read:recursion} and the construction of our chosen polynomial realization of test
vectors. We again use the integral representation obtained by
iterating \eqref{read:recursion}. The homogeneity of the polynomial,
its value at the minimum in Proposition~\ref{pub:saddle-main}, and
the scalar integral in Proposition~\ref{pub:scalar-calibration}
give the asymptotic ratio of the polynomial and scalar integrals.
Finally, the exact factor $K_{n,\epsilon}$ and the gamma factors in
$T_{n,\epsilon}$ show that the limit is precisely $1$.

For real $s>0$, make the change of variables
$x=u+\frac12\log(s/\pi)\mathbf v$.
By \eqref{pub:common-scale},
\[
 \mathcal V(x)=\frac s\pi\mathcal V(u),\qquad
 \mathcal A(x)=\left(\frac s\pi\right)^N\mathcal A(u),\qquad
 P_\epsilon(x)=\left(\frac s\pi\right)^\Omega P_\epsilon(u).
\]
Since $dx=du$ and $\mathcal F=\mathcal V-\log\mathcal A$,
\[
 I_{P_\epsilon,n}(s)
 =\left(\frac s\pi\right)^{Ns+\Omega}
   \int_{\mathbb R^D}P_\epsilon(u)e^{-s\mathcal F(u)}\,du,\qquad 
 I_{1,n}(s)=\left(\frac s\pi\right)^{Ns}
   \int_{\mathbb R^D}e^{-s\mathcal F(u)}\,du.
\]
Renaming $u$ as $x$ gives
\begin{equation}\label{pub:Laplace-ratio-integral}
 \frac{I_{P_\epsilon,n}(s)}{I_{1,n}(s)}
 =\left(\frac s\pi\right)^\Omega
 \frac{\int_{\mathbb R^D}P_\epsilon(x)e^{-s\mathcal F(x)}\,dx}
      {\int_{\mathbb R^D}e^{-s\mathcal F(x)}\,dx}.
\end{equation}
Put $\Xi=\mathcal F-\mathcal F(x_*)$ and
$J_s=\int_{\mathbb R^D}e^{-s\Xi}\,dx$.
By Proposition~\ref{pub:saddle-main}, $\Xi$ has the unique
minimum $0$ at $x_*$ and grows at least linearly at infinity.
Since $P_\epsilon$ grows at most exponentially, there exists
$s_0>0$ such that
$\int_{\mathbb R^D}(1+|P_\epsilon|)e^{-s_0\Xi}\,dx<\infty$.

Fix $\delta>0$ and let $B_\delta$ be the ball centered at $x_*$
of radius $\delta$. Then
$\eta_\delta=\inf_{\mathbb R^D\setminus B_\delta}\Xi>0$.
Choose $0<r<\delta$ such that $\Xi\le\eta_\delta/2$ on $B_r$.
For $s\ge s_0$, we have
\[
 J_s\ge\operatorname{vol}(B_r)e^{-s\eta_\delta/2},
 \qquad
 \int_{\mathbb R^D\setminus B_\delta}
 |P_\epsilon-P_\epsilon(x_*)|e^{-s\Xi}\,dx
 \le C e^{-(s-s_0)\eta_\delta}.
\]
Thus the contribution outside $B_\delta$, divided by $J_s$,
tends to zero. Inside $B_\delta$, the normalized error is at
most $\sup_{B_\delta}|P_\epsilon-P_\epsilon(x_*)|$. Hence
\[
 \limsup_{s\to+\infty}
 \left|\frac{\int_{\mathbb R^D}P_\epsilon e^{-s\Xi}\,dx}{J_s}
       -P_\epsilon(x_*)\right|
 \le \sup_{B_\delta}|P_\epsilon-P_\epsilon(x_*)|.
\]
By continuity of $P_\epsilon$, the right-hand side tends to zero
as $\delta\to0$, so the quotient converges to
$P_\epsilon(x_*)=n^{-n\epsilon/2}$ by Proposition~\ref{pub:saddle-main} as $s\to+\infty$.
The quotient in the above inequality is unchanged when $\Xi$ is replaced by $\mathcal F$,
so \eqref{pub:Laplace-ratio-integral} yields
\begin{equation}\label{pub:Laplace-limit}
 \frac{I_{P_\epsilon,n}(s)}{I_{1,n}(s)}
 =\left(\frac s\pi\right)^\Omega
   \bigl(n^{-n\epsilon/2}+o(1)\bigr).
\end{equation}

Using the exact factor
$K_{n,\epsilon}=\prod_{k=1}^{n-1}(-1)^{k\epsilon}\kappa_{k+1}$
from \eqref{read:recursion}, equations
\eqref{upper:period-integral}, \eqref{pub:I1-value}, and
\eqref{read:T} give
\begin{equation}\label{read:Q-ratio}
 Q(s)=\pi^\Omega
 \frac{I_{P_\epsilon,n}(s)}{I_{1,n}(s)}
 \frac{G_0(s)}{G_{\mathbf r,\epsilon}(s)}.
\end{equation}
Stirling's formula
\[
 \Gamma(x)=\sqrt{2\pi}\,x^{x-\frac12}e^{-x}
 \bigl(1+O(x^{-1})\bigr)
 \qquad(x\to+\infty)
\]
implies, for every fixed $a$,
$\Gamma(x+a)/\Gamma(x)=x^a(1+O(x^{-1}))$.
Applying this with $x=s/2$, $ns/2$, and $s$, respectively,
we obtain, as $s\to+\infty$,
\[
 \begin{aligned}
 \frac{G_{\mathbf r,\epsilon}(s)}{G_0(s)}
 &=
 \frac{\Gamma((s+\epsilon)/2)^n}{\Gamma(s/2)^n}
 \frac{\Gamma(ns/2)}{\Gamma(n(s+\epsilon)/2)}
 \prod_{i<j}\frac{\Gamma(s+b_{ij})}{\Gamma(s)}\\
 &=
 \left(\frac{s}{2}\right)^{n\epsilon/2}
 \left(\frac{ns}{2}\right)^{-n\epsilon/2}
 s^{\sum_{i<j}b_{ij}}\bigl(1+O(s^{-1})\bigr)\\&=n^{-n\epsilon/2}s^\Omega\bigl(1+O(s^{-1})\bigr),
 \end{aligned}
\]
where $\Omega=\sum_{i<j}b_{ij}$.
Combining this with \eqref{pub:Laplace-limit} in
\eqref{read:Q-ratio}, we obtain
\[
 Q(s)=
 \frac{\pi^\Omega(s/\pi)^\Omega
       \bigl(n^{-n\epsilon/2}+o(1)\bigr)}
      {n^{-n\epsilon/2}s^\Omega\bigl(1+O(s^{-1})\bigr)}
 =1+o(1),
\]
which proves $\lim_{s\to+\infty}Q(s)=1$.
For $d^\sharp=(-d_n,\ldots,-d_1)$ the adjacent differences
are $r_i^\sharp=r_{n-i}\ge1$, and the parity remains $\epsilon$.
The same argument therefore proves assertions (2) and (3) for
$Q^\vee$, completing Proposition~\ref{read:analytic-input} (3).


\begin{thebibliography}{99}

\bibitem{adrianTakeda2024local}
Moshe Adrian and Shuichiro Takeda.
\newblock On local {$L$}-factors for {Archimedean} {$\mathrm{GL}(n)$}.
\newblock {\em Journal of Algebra}, 653:133--155, 2024.

\bibitem{akmss2021}
U.~K. Anandavardhanan, Robert Kurinczuk, Nadir Matringe,
Vincent S\'echerre, and Shaun Stevens.
\newblock Galois self-dual cuspidal types and Asai local factors.
\newblock {\em Journal of the European Mathematical Society},
23(9):3129--3191, 2021.

\bibitem{anandavardhanan-matringe2017}
U.~K. Anandavardhanan and Nadir Matringe.
\newblock Test vectors for local periods.
\newblock {\em Forum Mathematicum}, 29(6):1245--1260, 2017.

\bibitem{anandavardhanan-rajan2005}
U.~K. Anandavardhanan and C.~S. Rajan.
\newblock Distinguished representations, base change, and reducibility
for unitary groups.
\newblock {\em International Mathematics Research Notices},
2005(14):841--854, 2005.

\bibitem{beuzart2021archimedean}
Rapha{\"e}l Beuzart-Plessis.
\newblock Archimedean theory and {$\epsilon$}-factors for the {Asai
  Rankin--Selberg} integrals.
\newblock In Werner M{\"u}ller, Sug~Woo Shin, and Nicolas Templier, editors,
  {\em Relative Trace Formulas}, Simons Symposia, pages 1--50. Springer, Cham,
  2021.
\newblock Section and theorem references follow
  \href{https://arxiv.org/abs/1812.00053v3}{arXiv:1812.00053v3}.

\bibitem{chen-cheng-theta}
Shih-Yu Chen and Yao Cheng.
\newblock Whittaker functions on {$\mathrm{GL}_n$} via theta lifting, 2026.
\newblock arXiv:2602.06315.

\bibitem{duran1989stieltjes}
Antonio~J. Dur\'an.
\newblock The Stieltjes moments problem for rapidly decreasing functions.
\newblock {\em Proceedings of the American Mathematical Society},
107(3):731--741, 1989.

\bibitem{flicker1988}
Yuval~Z. Flicker.
\newblock Twisted tensors and Euler products.
\newblock {\em Bulletin de la Soci\'et\'e Math\'ematique de France},
116(3):295--313, 1988.

\bibitem{flicker1993}
Yuval~Z. Flicker.
\newblock On zeroes of the twisted tensor {$L$}-function.
\newblock {\em Mathematische Annalen}, 297:199--219, 1993.

\bibitem{goldberg1994}
David Goldberg.
\newblock Some results on reducibility for unitary groups and local
Asai {$L$}-functions.
\newblock {\em Journal f\"ur die reine und angewandte Mathematik},
448:65--95, 1994.

\bibitem{grobner-harris-lapid2016}
Harald Grobner, Michael Harris, and Erez Lapid.
\newblock Whittaker rational structures and special values of the
Asai {$L$}-function.
\newblock In {\em Advances in the Theory of Automorphic Forms and
Their {$L$}-Functions}, volume 664 of {\em Contemporary Mathematics},
pages 119--134. American Mathematical Society, Providence, RI, 2016.

\bibitem{grobner-lin2021}
Harald Grobner and Jie Lin.
\newblock Special values of {$L$}-functions and the refined
Gan--Gross--Prasad conjecture.
\newblock {\em American Journal of Mathematics}, 143(3):859--937, 2021.

\bibitem{henniart2010}
Guy Henniart.
\newblock Correspondance de Langlands et fonctions {$L$} des carr\'es
ext\'erieur et sym\'etrique.
\newblock {\em International Mathematics Research Notices},
2010(4):633--673, 2010.

\bibitem{humphries-jo2024}
Peter Humphries and Yeongseong Jo.
\newblock Test vectors for archimedean period integrals.
\newblock {\em Publicacions Matem\`atiques}, 68(1):139--185, 2024.

\bibitem{ishii-stade}
Taku Ishii and Eric Stade.
\newblock Archimedean zeta integrals on {$\mathrm{GL}_n\times\mathrm{GL}_m$}
  and {$\mathrm{SO}_{2n+1}\times\mathrm{GL}_m$}.
\newblock {\em Manuscripta Mathematica}, 141(3--4):485--536, 2013.

\bibitem{jacquet2009}
Herv{\'e} Jacquet.
\newblock Archimedean {Rankin--Selberg} integrals.
\newblock In David Ginzburg, Erez Lapid, and David Soudry, editors, {\em
  Automorphic Forms and {$L$}-Functions II: Local Aspects}, volume 489 of {\em
  Contemporary Mathematics}, pages 57--172. American Mathematical Society,
  Providence, RI, 2009.

\bibitem{kable2004}
Anthony~C. Kable.
\newblock Asai {$L$}-functions and Jacquet's conjecture.
\newblock {\em American Journal of Mathematics}, 126(4):789--820, 2004.

\bibitem{kruse2021series}
Karsten Kruse.
\newblock Series representations in spaces of vector-valued functions
  via Schauder decompositions.
\newblock {\em Mathematische Nachrichten}, 294(2):354--376, 2021.

\bibitem{macdonald1995}
I.~G. Macdonald.
\newblock \emph{Symmetric Functions and Hall Polynomials}.
\newblock Second edition, Oxford Mathematical Monographs.
\newblock Clarendon Press, Oxford, 1995.

\bibitem{matringe2011}
Nadir Matringe.
\newblock Distinguished generic representations of {$\GL(n)$} over
{$p$}-adic fields.
\newblock {\em International Mathematics Research Notices},
2011(1):74--95, 2011.

\bibitem{matringe2013}
Nadir Matringe.
\newblock Essential Whittaker functions for {$\GL(n)$}.
\newblock {\em Documenta Mathematica}, 18:1191--1214, 2013.

\bibitem{raghuram2016critical}
A.~Raghuram.
\newblock Critical values of {Rankin--Selberg} {$L$}-functions for
  {$\mathrm{GL}_n\times\mathrm{GL}_{n-1}$} and the symmetric cube
  {$L$}-functions for {$\mathrm{GL}_2$}.
\newblock {\em Forum Mathematicum}, 28(3):457--489, 2016.

\bibitem{shahidi1985}
Freydoon Shahidi.
\newblock Local coefficients as Artin factors for real groups.
\newblock {\em Duke Mathematical Journal}, 52(4):973--1007, 1985.

\bibitem{shahidi1990}
Freydoon Shahidi.
\newblock A proof of Langlands' conjecture on Plancherel measures;
complementary series of {$p$}-adic groups.
\newblock {\em Annals of Mathematics} (2), 132(2):273--330, 1990.

\bibitem{shankman2018}
Daniel Shankman.
\newblock Local Langlands correspondence for Asai {$L$}-functions
and epsilon factors.
\newblock Preprint, 2018. \href{https://arxiv.org/abs/1810.11852}{arXiv:1810.11852}.

\bibitem{stein-weiss1971}
Elias~M. Stein and Guido Weiss.
\newblock \emph{Introduction to Fourier Analysis on Euclidean Spaces}.
\newblock Princeton Mathematical Series, volume~32.
\newblock Princeton University Press, Princeton, NJ, 1971.

\end{thebibliography}
\end{document}